\documentclass[11pt,a4paper]{amsart} 
\usepackage{amsthm,amsfonts,amsmath,calc,enumerate,hyperref,graphicx}
\usepackage[sort&compress,numbers]{natbib}
\usepackage[lmargin=35mm,rmargin=35mm,tmargin=30mm,bmargin=30mm,footskip=10mm]{geometry}

\theoremstyle{plain} 
\newtheorem{theorem}{Theorem}[section]
\newtheorem{lemma}[theorem]{Lemma}
\newtheorem{corollary}[theorem]{Corollary}
\newtheorem{proposition}[theorem]{Proposition}

\theoremstyle{definition}

\newcommand{\thmlabel}[1]{\label{thm:#1}}
\newcommand{\thmref}[1]{Theorem~\ref{thm:#1}}
\newcommand{\lemlabel}[1]{\label{lem:#1}}
\newcommand{\lemref}[1]{Lemma~\ref{lem:#1}}

\newcommand{\eqnlabel}[1]{\label{eqn:#1}}
\newcommand{\eqnref}[1]{\eqref{eqn:#1}}
\newcommand{\Eqnref}[1]{Equation~\eqref{eqn:#1}}
\newcommand{\threeeqnref}[3]{\eqref{eqn:#1}, \eqref{eqn:#2} and \eqref{eqn:#3}}
\newcommand{\figlabel}[1]{\label{fig:#1}}
\newcommand{\figref}[1]{Figure~\ref{fig:#1}}

\newcommand{\seclabel}[1]{\label{sec:#1}}
\newcommand{\secref}[1]{Section~\ref{sec:#1}}
\newcommand{\twosecref}[2]{Sections~\ref{sec:#1} and \ref{sec:#2}}
\newcommand{\corlabel}[1]{\label{cor:#1}}
\newcommand{\corref}[1]{Corollary~\ref{cor:#1}}
\newcommand{\proplabel}[1]{\label{prop:#1}}
\newcommand{\propref}[1]{Proposition~\ref{prop:#1}}

\newcommand{\ee}{\textsf{\textup{ee}}}

\newcommand{\Figure}[4][htb]{
  \begin{figure}[#1]
    \vspace*{1ex}
    \begin{center}#3\end{center} \vspace*{-1ex}
    \caption{\figlabel{#2}#4}
  \end{figure}}

\newcommand{\set}[1]{\ensuremath{\protect\{#1\}}}
\newcommand{\SET}[1]{\ensuremath{\protect\left\{#1\right\}}}

\newcommand{\MAXM}[1]{\ensuremath{\protect\max\SET{#1}}}
\newcommand{\minm}[1]{\ensuremath{\protect\min\set{#1}}}

\newcommand{\BRACKET}[1]{\ensuremath{\protect\left(#1\right)}}

\newcommand{\ceil}[1]{\ensuremath{\protect\lceil#1\rceil}}
\newcommand{\ang}[1]{\ensuremath{\protect\langle#1\rangle}}
\newcommand{\CEIL}[1]{\ensuremath{\protect\left\lceil#1\right\rceil}}
\newcommand{\FLOOR}[1]{\ensuremath{\protect\left\lfloor#1\right\rfloor}}
\newcommand{\floor}[1]{\ensuremath{\protect\lfloor#1\rfloor}}

\newcommand{\Oh}[1]{\ensuremath{\protect\mathcal{O}(#1)}}

\newcommand{\half}{\ensuremath{\protect\tfrac{1}{2}}}

\DeclareMathOperator{\outdeg}{outdeg}
\DeclareMathOperator{\indeg}{indeg}

\newcommand{\Ceil}[1]{\ensuremath{\big\lceil#1\big\rceil}}

\newcommand{\CeilFrac}[2]{\ensuremath{\big\lceil\frac{#1}{#2}\big\rceil}}
\newcommand{\CEILFRAC}[2]{\CEIL{\frac{#1}{#2}}}

\newcommand{\FloorFrac}[2]{\ensuremath{\big\lfloor\frac{#1}{#2}\big\rfloor}}
\newcommand{\FLOORFRAC}[2]{\FLOOR{\frac{#1}{#2}}}

\newcommand{\ceill}[2]{\ensuremath{\ceil{#1}_{#2}}}
\newcommand{\CEILL}[2]{\ensuremath{\CEIL{#1}_{#2}}}
\newcommand{\Ceill}[2]{\ensuremath{\big\lceil{#1}\big\rceil_{#2}}}

\newcommand{\mincover}[1][d]{\ensuremath{\textsf{\textup{mincover}}_{#1}}}
\newcommand{\mc}[1][d]{\ensuremath{\textsf{\textup{mincover}}_{#1}}}
\newcommand{\rmc}[1][d]{\ensuremath{\textsf{\textup{rmincover}}_{#1}}}
\newcommand{\minpart}[1][d]{\ensuremath{\textsf{\textup{minpart}}_{#1}}}
\newcommand{\cliques}[1][d]{\ensuremath{\textsf{\textup{numcliques}}_{#1}}}

\newcommand{\PP}{\ensuremath{\mathcal{P}}}
\newcommand{\F}{\ensuremath{\mathcal{F}}}
\newcommand{\D}{\ensuremath{\mathcal{D}}}
\newcommand{\C}{\ensuremath{\mathcal{C}}}

\newcommand{\X}{\ensuremath{\mathcal{X}}}
 
\newcommand{\Z}{\ensuremath{\mathbb{Z}}}

\newcommand{\RootedCompleteTree}[2]{\ensuremath{\protect\overrightarrow{\Gamma}_{\!\!#1,#2}}}
\newcommand{\CompleteTree}[2]{\ensuremath{\protect\Gamma_{\!#1,#2}}}

\renewcommand{\baselinestretch}{1.25}
\begin{document}

\title{Partitions and Coverings of Trees by Bounded-Degree Subtrees}

\author{David R. Wood} \thanks{Supported by a QEII Research Fellowship
  from the Australian Research Council; research initiated at
  Universitat Polit{\`e}cnica de Catalunya, where supported by the
  Marie Curie Fellowship MEIF-CT-2006-023865, and by the projects MEC
  MTM2006-01267 and DURSI 2005SGR00692.}  \address{\newline Department
  of Mathematics and Statistics\newline The University of
  Melbourne\newline Melbourne, Australia} \email{woodd@unimelb.edu.au}


\begin{abstract}
  This paper addresses the following questions for a given tree $T$
  and integer $d\geq2$: (1) What is the minimum number of degree-$d$
  subtrees that partition $E(T)$? (2) What is the minimum number of
  degree-$d$ subtrees that cover $E(T)$?  We answer the first question
  by providing an explicit formula for the minimum number of subtrees,
  and we describe a linear time algorithm that finds the corresponding
  partition. For the second question, we present a polynomial time
  algorithm that computes a minimum covering. We then establish a
  tight bound on the number of subtrees in coverings of trees with
  given maximum degree and pathwidth. Our results show that pathwidth
  is the right parameter to consider when studying coverings of trees
  by degree-3 subtrees. We briefly consider coverings of general
  graphs by connected subgraphs of bounded degree.
\end{abstract}

\maketitle
\tableofcontents
\newpage

\section{Introduction}
\seclabel{Intro}

This paper addresses the following questions, which are motivated by a
recent approach to drawing trees\footnote{We consider graphs $G$ that
  are simple and finite. A graph with one vertex is
  \emph{trivial}. Let $G$ be an (undirected) graph. The \emph{degree}
  of a vertex $v$ of $G$, denoted by $\deg_G(v)$, is the number of
  edges of $G$ incident with $v$. The minimum and maximum degrees of
  $G$ are respectively denoted by $\delta(G)$ and $\Delta(G)$. We say
  $G$ is \emph{degree-$d$} if $\Delta(G)\leq d$. Now let $G$ be a
  directed graph. Let $v$ be a vertex of $G$. The \emph{indegree} of
  $v$, denoted by $\indeg_G(v)$, is the number of incoming edges
  incident to $v$. The \emph{outdegree} of $v$, denoted by
  $\outdeg_G(v)$, is the number of outgoing edges incident to $v$. We
  say $G$ is \emph{outdegree-$d$} if $\outdeg_G(v)\leq d$ for every
  vertex $v$ of $G$.  A \emph{rooted tree} is a directed tree such
  that exactly one vertex, called the \emph{root}, has indegree
  $0$. It follows that every vertex except $r$ has indegree $1$, and
  every edge $vw$ of $T$ is oriented `away' from $r$; that is, if $v$
  is closer to $r$ than $w$, then $vw$ is directed from $v$ to $w$. If
  $r$ is a vertex of a tree $T$, then the pair $(T,r)$ denotes the
  rooted tree obtained by orienting every edge of $T$ away from $r$.}
developed in the companion paper \citep{HLW}.  For a given tree $T$
and integer $d\geq2$,
\begin{itemize}
\item what is the minimum number of degree-$d$ subtrees that partition
  $E(T)$?
\item what is the minimum number of degree-$d$ subtrees that cover
  $E(T)$?
\end{itemize}
Here a \emph{partition} of a graph $G$ is a set of connected subgraphs of $G$ such that every edge of $G$ is in exactly
one subgraph. A partition can also be thought of as a (non-proper)
edge-colouring, with one colour for each connected subgraph.  A
\emph{covering} of $G$ is a set of connected subgraphs of
$G$ such that every edge of $G$ is in at least one subgraph.  For
$d\geq2$, let $\minpart[d](G)$ be the minimum number of degree-$d$ connected subgraphs that partition $G$, and let
$\mincover[d](G)$ be the minimum number of degree-$d$
connected subgraphs that cover $G$. We emphasise that `trees' and
`subtrees' are necessarily connected.

In \secref{Partitioning} we answer the first question above. In
particular, we present an explicit formula for $\minpart[d](T)$, and
describe a linear time algorithm that finds the corresponding
partition (amongst other results).

The remainder of the paper addresses the second question above.
\secref{Paths} considers coverings of trees by paths (that is,
degree-2 subtrees). A tight bound on the number of paths is obtained,
amongst other combinatorial and algorithmic results.

Then \secref{TreeCovering} describes a polynomial time algorithm that
computes $\mincover[d](T)$ and the corresponding covering.
\secref{CompleteTrees} describes an example of this algorithm applied
to `complete' trees.

Then \secref{Caterpillars} studies minimum coverings of caterpillars
by degree-$d$ subtrees. Again tight bounds on the number of subtrees
are obtained. Coverings of caterpillars provide a natural precursor to
the results in \twosecref{RootedPathwidth}{UnrootedPathwidth}. These
sections establish tight upper bounds on the number of covering
subtrees in terms of the pathwidth and maximum degree of the
tree. Essentially, these results show that pathwidth is the right
parameter to consider when studying coverings of trees by degree-3
subtrees.

Finally, \secref{General} studies coverings of general and planar
graphs by connected subgraphs of bounded degree. While this problem
has not previously been explicitly studied, in the case $d=2$, a
related concept has been extensively studied.  Harary defined the
\emph{pathos} or \emph{path number} of a graph $G$, denoted by $p(G)$,
to be the minimum number of paths that partition $E(G)$; see
\citep{Lovasz-Covering,Nieminen75,Harary70,DK-DM00,HS72a,SCJ70,SJC72,Donald-JGT80}. \citet{HS72a}
defined the \emph{unrestricted path number} of a graph $G$, denoted by
$p^*(G)$, to be the minimum number of paths that cover $G$.  Since
every cycle is the union of two disjoint paths,
\begin{align*}
  \minpart[2](G)\leq p(G)\leq 2\cdot\minpart[2](G)\text{, and}\\
  \mincover[2](G)\leq p^*(G)\leq 2\cdot \mincover[2](G)\enspace,
\end{align*}
where the lower bounds on $p(G)$ and $p^*(G)$ become equalities if $G$
is a tree. Also concerning the $d=2$ case, Gallai conjectured that
$p(G)\leq\Ceil{\frac{n+1}{2}}$ for every connected graph $G$ with $n$
vertices. While this conjecture remains unsolved,
\citet{Lovasz-Covering} proved that $\minpart[2](G)\leq\Ceil{\frac{n}{2}}$
for every (not neccessarily connected) graph $G$;
also see \citep{Fan-JCTB05,Donald-JGT80,Pyber-JCTB96,DK-DM00,Fan-PathCovers-JGT}.

\section{Partitioning Trees}
\seclabel{Partitioning}

This section considers partitions of (the edge-set of) a tree into
bounded-degree subtrees\footnote{Note that the question for
  bounded-degree subforests is easily answered. A straightforward
  inductive argument proves that for every degree-$\Delta$ tree $T$,
  there is a partition of $E(T)$ into \ceil{\frac{\Delta}{d}}
  degree-$d$ subforests. This is just an edge-colouring such that
  every vertex is incident to at most $d$ monochromatic edges; see
  \citep{HdW94} for analogous results for general graphs. The bound of
  \ceil{\frac{\Delta}{d}} is best possible for every tree $T$, and
  there is a linear-time algorithm to compute the partition.}. First
we prove a formula for $\minpart[d](T)$. Interestingly, it only
depends on the degrees modulo $d$.

\begin{theorem}
  \thmlabel{MinPart} 
Let $T$ be a non-trivial tree with $n\geq 2$ vertices, and
  let $d\geq2$. Define
$$n_i:=\left|\left\{v\in V(T):\CEILFRAC{\deg(v)}{d}=\frac{\deg(v)+i}{d} \right\}\right|$$ for
$i\in[0,d-1]$. Then
$$\minpart[d](T)
=1+\sum_{v\in V(T)}\BRACKET{\CEIL{\frac{\deg(v)}{d}}-1}
=1+\frac{2(n-1)}{d}-n+\sum_{i=0}^{d-1}\frac{i\cdot n_i}{d} \enspace.$$
Moreover, there is a linear-time algorithm to compute $\minpart[d](T)$
and a corresponding partition.
\end{theorem}

\begin{proof}
  First note that
  \begin{align*}
    1+\sum_{v\in V(T)}\BRACKET{\CEIL{\frac{\deg(v)}{d}}-1}
    =&\;1+\sum_{v\in  V(T)}\BRACKET{\frac{\deg(v)}{d}-1}+\sum_{i=0}^{d-1}\frac{i\cdot n_i}{d}\\
    =&\;1+\frac{2(n-1)}{d}-n+\sum_{i=0}^{d-1}\frac{i\cdot
      n_i}{d}\enspace.
  \end{align*}
  Thus it suffices to prove the first equality.  We proceed by
  induction. In the base case, $T=K_2$ and $\minpart[d](T)=1$. Now
  assume that $T$ has at least three vertices. Thus $T$ has a set of
  leaves $S$ with a common neighbour $w$, such that $w$ is a leaf in
  $T-S$. By induction,
$$\minpart[d](T-S)=1+\sum_{v\in~V(T-S)}\BRACKET{\CEIL{\frac{\deg_{T-S}(v)}{d}}-1}\enspace.$$

To extend the partition of $T-S$, the colour that is assigned to the
edge in $T-S$ incident to $w$ can be assigned to $d-1$ edges incident
to $S$. There are $(\deg(w)-1)-(d-1)=\deg(w)-d$ remaining leaf edges
incident to $w$. These leaf edges can be partitioned into
$\Ceil{\frac{\deg(w)-d}{d}}=\Ceil{\frac{\deg(w)}{d}}-1$ stars rooted
at $w$, each with at most $d$ edges. This defines a partition of $T$
into $\minpart[d](T-S)+\Ceil{\frac{\deg(w)}{d}}-1$ subtrees, which
equals the claimed upper bound on $\minpart[d](T)$ since leaves do not
contribute to the summation.

It is easy to convert this proof into a linear-time algorithm. Here is
a sketch. Root $T$ at a vertex $r$, and partition the vertex sets
according to their distance from $r$ (using BFS). The set $S$ is
simply a maximal set of leaves at maximum distance $d$ from the root
with a common neighbour (at distance $d-1$). The partition is then
easily computed.

This bound on $\minpart[d](T)$ is optimal since at most $d-1$ edges
incident to $S$ can share the same colour as the edge in $T-S$
incident to $w$, and at least $\Ceil{\frac{\deg(w)-d}{d}}$ colours not
used in $T-S$ must be introduced on the remaining leaf edges.
\end{proof}

Note that \thmref{MinPart} with $d=2$ reduces to the following result
by \citet{SCJ70}:

\begin{corollary}[\citep{SCJ70}] 
  For every non-trivial tree $T$, $\minpart[2](T)$ equals half the
  number of odd-degree vertices.
\end{corollary}

\thmref{MinPart} also implies:

\begin{corollary}
  For every integer $d\geq2$ and tree $T$ with $n\geq2$ vertices,
$$\minpart[d](T)\leq1+\frac{n-2}{d},$$
with equality if and only if $\deg(v)\equiv1\pmod{d}$ for every vertex
$v$.
\end{corollary}


A degree-$d$ subtree $X$ of a tree $T$ is \emph{degree-$d$ maximal} if
no edge of $T$ can be added to $X$ to obtain a new degree-$d$
subtree. Observe that $X$ is degree-$d$ maximal if and only if
$\deg_X(v)=\minm{d,\deg_T(v)}$ for every vertex $v$ of $X$. In
particular, $v$ is leaf in $X$ if and only if $v$ is a leaf in
$T$. Clearly every degree-$d$ subtree is contained in a maximal
degree-$d$ subtree.

\begin{proposition}
  Every degree-$d$ maximal subtree $S$ of a tree $T$ is in a minimum
  partition of $T$ into degree-$d$ subtrees.
\end{proposition}

\begin{proof}
  For each vertex $v$ of $S$, let $T_v$ be the component of $T-E(S)$
  that contains $v$. Note that $T_v$ is trivial (that is, $v$ is the
  only vertex in $T_v$) if every edge incident to $v$ is in $S$. If
  $T_v$ is non-trivial then $\deg_S(v)=d$, as otherwise $S$ is not
  maximal. Let $N$ be the set of vertices $v$ in $S$ such that $T_v$
  is non-trivial. Taking $S$ with a minimum partition of each $T_v$
  into degree-$d$ subtrees (where $T_v$ is non-trivial) gives a
  partition of $T$ into
$$1 + \sum_{v\in N}\minpart[d](T_v)$$
parts. By \thmref{MinPart},
\begin{align*}
  &1 + \sum_{v\in N}\minpart[d](T_v)\\
  \leq\;& 1 + \sum_{v\in N}
  \BRACKET{1+\sum_{x\in V(T_v)}\BRACKET{\CEIL{\frac{\deg_{T_v}(x)}{d}}-1}} \\
  =\;& 1 + \sum_{v\in N}
  \BRACKET{1+\BRACKET{\CEIL{\frac{\deg_{T}(v)-d}{d}}-1}+
    \sum_{x\in V(T_v-v)}\BRACKET{\CEIL{\frac{\deg_{T_v}(x)}{d}}-1}} \\
  =\;& 1 + \sum_{v\in N}
  \BRACKET{\CEIL{\frac{\deg_{T}(v)}{d}}-1+\sum_{x\in
      V(T_v-v)}\BRACKET{\CEIL{\frac{\deg_{T_v}(x)}{d}}-1}} \enspace.
\end{align*}
Observe that if $x$ is in $V(T)-V(S)$, then $x$ is in exactly one
subtree $T_v$, and this $T_v$ is non-trivial, and
$\deg_{T_v}(x)=\deg_T(x)$. Thus
\begin{align*}
  1 + \sum_{v\in N}\minpart[d](T_v) & \leq 1 + \sum_{x\in
    V(T)-V(S)\cup N}\BRACKET{\CEIL{\frac{\deg_{T}(x)}{d}}-1} \enspace.
\end{align*}
If $T_v$ is trivial then $\deg_S(v)=\deg_T(v)\leq d$, as otherwise $S$
is not maximal. Thus $\CeilFrac{\deg_T(v)}{d}-1=0$. Hence
\begin{align*}
  1 + \sum_{v\in N}\minpart[d](T_v)& \leq 1 + \sum_{x\in
    V(T)}\BRACKET{\CEIL{\frac{\deg_{T}(x)}{d}}-1} \enspace.
\end{align*}
Hence this partition is minimum by \thmref{MinPart}.
\end{proof}

\section{Coverings by Paths}
\seclabel{Paths}

This section studies coverings of trees by degree-2 subtrees.  Since a
subtree is degree-2 if and only if it is a path, $\minpart[2](T)$ is
the minimum number of paths that cover $T$.  Since each path covers at
most two leaves, if $T$ has $\ell$ leaves then at least
$\Ceil{\frac{\ell}{2}}$ paths are required. \citet{HS72a} first proved
the converse:

\begin{theorem}[\citep{HS72a,AKG-DM96}]
  \thmlabel{AKG} The minimum number of paths that cover a tree with
  $\ell$ leaves is $\Ceil{\frac{\ell}{2}}$.\qed
\end{theorem}

In this section we explore this topic further. First, we investigate
the total number of edges in a covering of a tree by the minimum
number of paths. Let $P$ be a path in a tree $T$. Then $P$ is
\emph{leafy} if both endpoints of $P$ are leaves in $T$. And $P$ is a
\emph{pendant path} if one endpoint of $P$ is a leaf in $T$, the other
endpoint of $P$ has degree at least 3 in $T$, and every internal
vertex in $P$ has degree 2 in $T$.

\begin{theorem}
  \thmlabel{SizePathCovering} Let $T$ be a tree with $n$ vertices and
  $\ell$ leaves. Then $T$ has a covering by $\Ceil{\frac{\ell}{2}}$
  paths with $2n-2-\ell$ edges in total.  Moreover, for infinitely
  many trees $T$, every covering of $T$ by $\Ceil{\frac{\ell}{2}}$
  paths has at least $2n-2-\ell$ edges in total.
\end{theorem}

\begin{proof}
  Let \PP\ be a set of $\Ceil{\frac{\ell}{2}}$ paths that cover $T$
  and minimise the total number of edges. By \thmref{AKG} this is well
  defined. Each leaf is in exactly one path in \PP\ (by the minimality
  of \PP). Each path in \PP\ covers at most two leaves. Thus every
  path in \PP\ is leafy, except if $\ell$ is odd, in which case, one
  path in \PP\ is a pendant path, and every other path is leafy. Also
  note that this pendant path shares no edge in common with another
  path in \PP\ (again by the minimality of \PP).

  Suppose on the contrary that some edge $e$ of $T$ is in three
  distinct paths $P_1,P_2,P_3\in\PP$. Thus each $P_i$ is leafy. Let
  $R$ be an edge-maximal path in $T$ that contains $e$, such that
  every internal vertex of $R$ has degree exactly 2. (It is possible
  that $e$ is the only edge in $R$.)\ Observe that $R$ is contained in
  each of $P_1,P_2,P_3$. Let $v$ and $w$ be the endpoints of $R$.  Let
  $T_v$ and $T_w$ be the two component subtrees of $T-E(R)$, such that
  $v$ is in $T_v$ and $w$ is in $T_w$. Let $Q_v:=T_v\cap(P_1\cup P_2)$
  and $Q_w:=T_w\cap(P_1\cup P_2)$. Thus $Q_v,Q_w,P_3$ are three paths
  that cover $P_1\cup P_2\cup P_3$, but with $2|E(R)|$ less edges in
  total. Hence, replacing $P_1$ and $P_2$ by $Q_v$ and $Q_w$ in \PP,
  produces a covering of $T$ by $\Ceil{\frac{\ell}{2}}$ paths with
  less edges in total. This contradiction proves that every edge in
  $T$ is in at most two paths in \PP. Hence \PP\ has at most $2(n-1)$
  edges in total. Moreover, no leaf edge of $T$ is in two paths in
  \PP. Thus \PP\ has at most $2n-2-\ell$ edges in total.

  To prove that the lower bound, let $T_0$ be a subdivision of the
  $p$-leaf star. Say $T_0$ has $q$ vertices. Let $v_1,\dots,v_p$ be
  the leaves of $T_0$. Let $T$ be the tree obtained from $T_0$ by
  adding two leaves $u_i$ and $w_i$ adjacent to $v_i$, for each
  $i\in\{1,\dots,p\}$. So $T$ has $n:=q+2p$ vertices and $\ell:=2p$
  leaves. Let \PP\ be a set of $p$ paths that cover $T$. Each path in
  \PP\ connects two leaves of $T$. But no path $u_iv_iw_i$ is in \PP\
  (otherwise $|\PP|>p$). Hence each edge in $T_0$ is in at least two
  paths in \PP. It follows that \PP\ has at least
  $2(q-1)+2p=2n-\ell-2$ edges in total.
\end{proof}

We now sharpen \thmref{SizePathCovering} for trees with an even number
of leaves.  Let $L(T)$ be the set of leaves in a tree $T$.  Let $vw$
be an edge in $T$.  Let $T_v$ and $T_w$ be the components of $T-vw$
that respectively contain $v$ and $w$. Then $vw$ is said to be
\emph{even--even} if $|T_v\cap L(T)|$ and $|T_w\cap L(T)|$ are both
even. Let $\ee(T)$ be the number of even--even edges in $T$.

\begin{theorem}
  \thmlabel{SizePathCoveringEven} Let $T$ be a tree with $n$ vertices
  and $\ell$ leaves, where $\ell$ is even. Then $T$ has a covering by
  $\frac{\ell}{2}$ paths with $n-1+\ee(T)$ edges in total, and this
  covering can be computed in $O(n)$ time.  Moreover, every covering
  of $T$ by $\frac{\ell}{2}$ paths has at least $n-1+\ee(T)$ edges.
\end{theorem}

\begin{proof}
  Let $G$ be the multigraph obtained from $T$ by adding a second copy
  of each even--even edge in $T$.  Consider a non-leaf vertex $v$ of
  $T$. For each neighbour $w$ of $v$, let $T_w$ be the component of
  $T-v$ that contains $w$. Since $\ell$ is even, there are an even
  number of neighbours $w$ of $v$ such that $|V(T_w)\cap L(T)|$ is
  odd. If $|V(T_w)\cap L(T)|$ is even then $vw$ is doubled in
  $G$. Hence $v$ has even degree in $G$. Arbitrarily pair the edges
  incident to $v$ in $G$. By following sequences of paired edges in
  $G$ we obtain the desired covering of $T$. Since $G$ has
  $n-1+\ee(T)$ edges, the total number of edges in the paths is
  $n-1+\ee(T)$.

  The numbers $|V(T_w)\cap L|$ can be computed in a single traversal
  of the tree. The pairing step at each vertex $v$ can be implemented
  in $O(\deg(v))$ time, which is $O(n)$ in total. To output the paths,
  choose a leaf vertex $v$, find the maximal path $P$ starting at $v$
  in $G$, delete the edges in $P$ from $G$, and repeat. This algorithm
  can be easily implemented in $O(n)$ time.

  We now prove the `moreover' claim. Let \PP\ be a set of
  $\frac{\ell}{2}$ paths that cover $T$.  Each leaf is in some path in
  \PP, and each path in \PP\ covers at most two leaves. Thus each leaf
  is in exactly one path in \PP, and the endpoints of each path in
  \PP\ are leaves.

  Consider an edge $vw$ that appears in only one path $P\in\PP$. Then
  $P$ connects a leaf in $T_v\cap L(T)$ with a leaf in $T_w\cap
  L(T)$. Every other path in \PP\ is contained in $T_v$ or in
  $T_w$. Each such path has both endpoints in $T_v\cap L(T)$ or both
  endpoints in $T_w\cap L(T)$. Thus $|T_v\cap L(T)|$ and $|T_w\cap
  L(T)|$ are both odd. Hence each even-even edge is in at least two
  paths in $\PP$. Therefore the total number of edges in \PP\ is at
  least $n-1+\ee(T)$ edges.
\end{proof}

We now show that there is a minimal covering of $T$ by paths with
other properties.

\begin{proposition}
  Let $T$ be a tree with $\ell$ leaves. Then $T$ has a covering by
  $\Ceil{\frac{\ell}{2}}$ paths that have a vertex in common.
\end{proposition}

\begin{proof}
  Let \PP\ be a set of $\Ceil{\frac{\ell}{2}}$ paths in $T$ that cover
  every leaf in $T$, and with maximum total size. Suppose that there
  are disjoint paths $P$ and $Q$ in \PP. Let $R$ be a minimal path in
  $T$ between $P$ and $Q$. Let $v$ and $w$ be the endpoints of $R$,
  where $v$ is in $P$ and $w$ is in $Q$. Thus $P$ is the union of two
  paths $P_1$ and $P_2$ whose intersection is $v$. And $Q$ is the
  union of two paths $Q_1$ and $Q_2$ whose intersection is
  $w$. Replace $P$ and $Q$ in \PP\ by $P_1\cup R\cup Q_1$ and $P_2\cup
  R\cup Q_2$. We obtain a set of $\Ceil{\frac{\ell}{2}}$ paths with
  greater total size than \PP. This contradiction proves that the
  paths in \PP\ are pairwise intersecting. By the Helly property of
  subtrees of a tree, the paths in \PP\ have a vertex $v$ in common.
\end{proof}

\begin{lemma}
  \lemlabel{CoverEveryEdge} If \PP\ is a set of paths in a tree $T$
  that cover every leaf, and some vertex $v$ is in every path in \PP,
  then \PP\ covers every edge.
\end{lemma}

\begin{proof}
  Suppose on the contrary that some edge $e$ is not covered by
  \PP. Let $T_1$ and $T_2$ be the components of $T-e$. Without loss of
  generality, $v$ is in $T_1$. Let $P$ be a path in \PP\ that covers
  some leaf of $T$ contained in $T_2$. Then $v$ is not in $P$. This
  contradiction proves that \PP\ covers every edge.
\end{proof}

We now characterise those vertices $v$ for which there is a minimal
covering by paths all containing $v$.

\begin{theorem}
  \thmlabel{Centroids} Let $T$ be a tree with $\ell$ leaves.  Let $v$
  be a vertex of $T$.  Then $T$ has a covering by
  $\Ceil{\frac{\ell}{2}}$ paths each containing $v$ if and only if
  $|V(T')\cap L(T)|\leq\Ceil{\frac{\ell}{2}}$ for every component $T'$
  of $T-v$.
\end{theorem}

\begin{proof}
  Let $G$ be the graph with vertex set $L(T)$ where two leaves $x$ and
  $y$ are adjacent in $G$ if and only if $x$ and $y$ are in distinct
  components of $T-v$. Thus $G$ is a complete $d$-partite graph, where
  $d=\deg(v)$. Each colour class of $G$ consists of the leaves in
  $L(T)$ that are in a single component of $T-v$. Each pair of leaves
  in distinct components of $T-v$ are the endpoints of a path through
  $v$. In this way, each edge of $G$ corresponds to a leafy path in
  $T$. The paths in a covering of $T$ can be assumed to be leafy
  paths. Thus \lemref{CoverEveryEdge} implies that $T$ has a covering
  by $\Ceil{\frac{\ell}{2}}$ paths each containing $v$ if and only $G$
  contains a matching of $\floor{\frac{\ell}{2}}$ edges, or
  equivalently if $G$ contains $\ceil{\frac{\ell}{2}}$ disjoint
  $(\leq2)$-cliques. A result of \citet{Sitton-EJUM} (see \lemref{CMG}
  for a generalisation) implies that this property holds if and only
  if each colour class in $G$ has at most $\Ceil{\frac{\ell}{2}}$
  vertices, which is equivalent to saying that $|V(T')\cap
  L(T)|\leq\Ceil{\frac{\ell}{2}}$ for every component $T'$ of $T-v$.
\end{proof}

\begin{theorem}
  \thmlabel{CentroidPath} Let $T$ be a tree with $\ell$ leaves.  Let
  $C$ be the set of vertices $v$ in $T$ such that $T$ has a covering
  by $\Ceil{\frac{\ell}{2}}$ paths each containing $v$. Then $C$
  induces a non-empty path. Moreover, every internal vertex has degree
  $2$ in $T$, unless $\ell$ is odd, in which case $C$ may have exactly
  one internal vertex $v$ with degree exactly $3$, and $v$ is the
  endpoint of a pendant path.
\end{theorem}

\begin{proof}
  Let $L$ be the set of leaves in $T$.  Let $u,w\in C$. By
  \thmref{Centroids}, the number of leaves of $T$ in each component of
  $T-u$ or of $T-w$ is at most $\Ceil{\frac{\ell}{2}}$.  Let $v$ be a
  vertex on the $uw$-path in $T$. If $T'$ is a component of $T-v$,
  then $T'$ is contained in a component of $T-u$ or a component of
  $T-w$. Thus the number of leaves of $T$ in $T'$ is at most
  $\Ceil{\frac{\ell}{2}}$. By \thmref{Centroids}, $v\in C$. Hence
  $T[C]$ is connected.

  Suppose that $C$ contains a vertex $v$ with three neighbours
  $w_1,w_2,w_3$ in $C$. Let $T_i$ be the component of $T-v$ containing
  $w_i$. Without loss of generality, $|T_1\cap L|\leq |T_2\cap L|\leq
  |T_3\cap L|$. Thus $|(T_2\cup T_3)\cap L|\geq\frac23 \ell$. But
  $T_2\cup T_3$ is contained in a component of $T-w_1$, implying
  $w_1\not\in C$ by \thmref{Centroids}. This contradiction proves that
  $T[C]$ is a path.

  Now suppose that $C$ has an internal vertex $v$ such that
  $\deg_T(v)\geq3$. Let $u$ and $w$ be the neighbours of $v$ in $C$.
  Let $L_u$ be the set of leaves of $T$ in the component of $T-uv$
  that contains $u$.  Let $L_w$ be the set of leaves of $T$ in the
  component of $T-wv$ that contains $w$.  Let $L_v$ be the number of
  leaves of $T$ in the component of $T-\{uv,vw\}$ that contains
  $v$. Thus $\ell=|L_u|+|L_v|+|L_w|$. Since $\deg_T(v)\geq3$, we have
  $|L_v|\geq1$. Without loss of generality, $|L_u|\leq |L_w|$. Thus
  $|L_v|+|L_w|=\ell-|L_u|\geq\ell-|L_w|$, implying $\ell\leq
  2|L_w|+|L_v|\leq 2|L_w|+2|L_v|-1$. Hence $\half(\ell+1)\leq |L_w\cup
  L_v|$. Since $L_v\cup L_w$ is contained in a component of $T-u$, by
  \thmref{Centroids}, $\Ceil{\frac{\ell}{2}}\geq |L_w\cup L_v|$, which
  is a contradiction if $\ell$ is even. If $\ell$ is odd then $L_v=1$,
  and by a similar agument, $v$ is the only internal vertex of $C$
  with degree at least $3$ in $T$, and $v$ is the endpoint of a
  pendant path.
\end{proof}

\thmref{CentroidPath} says that the set of vertices $v$ for which $T$
has a minimal covering by paths each containing $v$ is somewhat like
the centroid of $T$, where we measure the `weight' of a component of
$T-v$ by the number of leaves in it rather than the number of
vertices.

Finally in this section, we consider the problem of covering a given
tree with a small number of subtrees, each with at most $d$
leaves. Covering by paths corresponds to the $d=2$ case. The next
result thus generalises \thmref{AKG} (and with a completely different
proof).

\begin{theorem}
  \thmlabel{FewLeaves} For every integer $d\geq2$ and for every tree
  $T$ with $\ell$ leaves, the minimum number of subtrees, each with at
  most $d$ leaves, that cover $T$ is $\Ceil{\frac{\ell}{d}}$.
\end{theorem}

\begin{proof} The lower bound is immediate. We prove the upper bound
  by induction on $\ell$. Clearly we can assume that $T$ has no vertex
  of degree $2$. For $S\subseteq L(T)$, let $T[S]$ denote the subtree
  of $T$ consisting of the union of all leafy paths in $T$ whose
  endpoints are both in $S$. Note that $T[S]$ has $|S|$ leaves.  Let
  $X$ be the set of vertices of $T$ that have degree at least $3$ and
  are adjacent to at least one leaf.

  First suppose that $|X|\geq d$. For each of $d$ vertices $x\in X$,
  choose one leaf incident to $x$. We obtain a set $L_0$ of $d$ leaves
  of $T$, such that no two vertices in $L_0$ have a common neighbour
  (in $X$). Since vertices in $X$ have degree at least $3$, $T-L_0$ is
  a tree with $\ell-d$ leaves. By induction, there is a covering of
  $T-L_0$ by $\Ceil{\frac{\ell}{d}}-1$ subtrees, each with at most $d$
  leaves. With $T[L_0]$, we obtain the desired covering of $T$ (since
  every edge in $T-V(T-L_0)$ is adjacent to a vertex in $L_0$, and is
  thus in $T[L_0]$).

  Now assume that $|X|\leq d$. If $\ell<d$, then the result is
  trivial. Otherwise $\ell\geq d$. Let $L_0$ be a set of $d$ leaves,
  such that each vertex in $X$ is adjacent to at least one leaf in
  $L_0$. Since $T[L[T]]=T$ and every leaf has a neighbour in common
  with some leaf in $L_0$, every non-leaf edge of $T$ is in
  $T[L_0]$. Arbitrarily partition the $\ell-d$ leaves in $S\setminus
  L_0$ into sets $\{L_i:1\leq i\leq \Ceil{\frac{\ell}{d}}-1\}$ such
  that each $|L_i|\leq d$. Hence $\{T[L_i]:0\leq i\leq
  \Ceil{\frac{\ell}{d}}-1\}$ is the desired covering of
  $T$. \end{proof}

\section{An Algorithm for Covering Trees}
\seclabel{TreeCovering}


This section describes a polynomial time algorithm to determine a
minimum covering of a tree $T$ by degree-$d$ subtrees. Since a subtree
is degree-2 if and only if it is a path, the results in this section
with $d=2$ generalise some of the results from \secref{Paths}.

It will be convenient to consider the following more general
scenario. Let $G$ be a connected graph.  A \emph{binding function} of
$G$ is a function $f:V(G)\rightarrow\{2,3,4,\dots\}$. A subgraph $X$
of $G$ is \emph{$f$-bound} if $\deg(v)\leq f(v)$ for every vertex $v$
of $X$. A covering \C\ of $G$ is \emph{degree-$f$} if every subgraph
$X\in\C$ is $f$-bound. For an integer $d\geq2$, a \emph{$d$-covering}
of $G$ is a degree-$f$ covering of $G$, where $f(v):=d$ for each
vertex $v$ of $G$. Let $\mincover[f](G)$ be the minimum cardinality of
a degree-$f$ covering of $G$. An $f$-bound subgraph $X$ of $G$ is
\emph{$f$-maximal} if no edge of $G-E(X)$ can be added to $X$ to
obtain a new $f$-bound subgraph.

This section describes a polynomial time algorithm to determine
$\mincover[f](T)$ and the corresponding degree-$f$ covering for any
given tree $T$ and binding function $f$. Observe that a subtree $X$ of
$T$ is $f$-maximal if and only if $\deg_X(v)=\minm{f(v),\deg_G(v)}$
for every vertex $v$ of $X$. In particular, $v\in V(X)$ is a leaf of
$X$ if and only if $v$ is a leaf of $T$.

\begin{lemma}
  \lemlabel{Split} Let $v$ be a vertex of a non-trivial connected
  graph $G$. Then $G$ contains connected subgraphs $G_1$ and $G_2$
  such that $G_1\cup G_2=G$ and $V(G_1)\cap V(G_2)=\{v\}$ and
  $\deg_{G_1}(v)\leq\max\{\deg_{G}(v)-1,1\}$ and
  $\deg_{G_2}(v)\leq\max\{\deg_G(v)-1,1\}$.
\end{lemma}

\begin{proof}
  First suppose there is a bridge edge $vw$ incident to $v$. Let $A$
  be the connected component of $G-vw$ that contains $w$. Then
  $G_1:=G[V(A)\cup\{v\}]$ and $G_2:=G-A$ satisfy the claim.  Now
  assume that no edge incident to $v$ is a bridge. Let $vw$ be an edge
  incident to $v$. Then $G_1:=G[\{v,w\}]$ and $G_2:=G-vw$ satisfy the
  claim.
\end{proof}

\begin{lemma}
  \lemlabel{GrowingGraph} Let $f$ be a binding function of a connected
  graph $G$. Let $s:=\mincover[f](G)$. Then $G$ has a degree-$f$
  covering by $s$ $f$-maximal subgraphs that are pairwise
  intersecting.
\end{lemma}

\begin{proof}
  Let $\{U_1,\dots,U_s\}$ be a degree-$f$ covering such that
  $\sum_i|E(U_i)|$ is maximum. Then each $U_i$ is non-trivial and $f$-maximal.
  Suppose on the contrary that $V(U_i)\cap V(U_j)=\emptyset$ for some
  $i,j$. Let $P$ be a shortest path between $U_i$ and $U_j$ in $G$,
  where $V(P)\cap V(U_i)=\{v\}$ and $V(P)\cap V(U_j)=\{w\}$.

  By \lemref{Split}, $U_i$ contains connected subgraphs $A_1$ and
  $A_2$ such that $A_1\cup A_2=U_i$ and $V(A_1)\cap V(A_2)=\{v\}$, and
  $\deg_{A_1}(v)\leq\max\{\deg_{U_i}(v)-1,1\}$ and
  $\deg_{A_2}(v)\leq\max\{\deg_{U_i}(v)-1,1\}$.  Similarly, $U_j$
  contains connected subgraphs $B_1$ and $B_2$ such that $B_1\cup
  B_2=U_j$ and $V(B_1)\cap V(B_2)=\{w\}$, and
  $\deg_{B_1}(v)\leq\max\{\deg_{U_j}(v)-1,1\}$ and
  $\deg_{B_2}(v)\leq\max\{\deg_{U_j}(v)-1,1\}$.

  Observe that $A_1\cup P\cup B_1$ and $A_2\cup P\cup B_2$ are
$f$-bound subgraphs of $G$ (since the degree of $v$ in
  each subgraph is at most $\max\{\deg_{U_i}(v),2\}\leq f(v)$, the
  degree of $w$ in each subgraph is at most
  $\max\{\deg_{U_j}(w),2\}\leq f(w)$, and for each internal vertex $z$
  in $P$, the degree of $z$ in each subgraph is at most $2\leq f(z)$).
  Since $A_1\cup A_2=U_i$ and $B_1\cup B_2=U_j$, replacing $U_i$ and
  $U_j$ by $A_1\cup P\cup B_1$ and $A_2\cup P\cup B_2$ gives a
  degree-$f$ covering of $G$ with greater total size than
  $U_1,\dots,U_s$. This contradiction proves that $V(U_i)\cap
  V(U_j)\neq\emptyset$ for $i,j\in\{1,\dots,r\}$.
\end{proof}

By the Helly property of subtrees of a tree, \lemref{GrowingGraph}
implies:

 \begin{lemma}
   \lemlabel{Growing} Let $f$ be a binding function of a tree $T$. Let
   $s:=\mincover[f](T)$. Then $G$ has a degree-$f$ covering by $s$
   $f$-maximal subtrees that have a vertex in common.
 \end{lemma}

 \lemref{Growing} implies that to find a minimum degree-$f$ covering
 of a tree, it suffices to consider degree-$f$ coverings that have a
 common vertex in every subtree. It is therefore convenient to
 consider the following more general covering problem. Recall that in
 a rooted tree $T$ the edges are oriented away from the root vertex
 $r$. A \emph{rooted covering} of $T$ is a covering \C\ of $T$ such
 that $r$ is in every subtree in \C. A \emph{binding function} of $T$
 is a function $f:V(T)\rightarrow\mathbb{Z}^+$. A rooted covering
 $\mathcal{C}$ of $T$ is \emph{outdegree-$f$} if $\outdeg_X(v)\leq
 f(v)$ for every vertex $v$ in every subtree $X\in\mathcal{C}$. For an
 integer $d\geq1$, a \emph{degree-$d$} rooted covering of $T$ is an
 outdegree-$f$ rooted covering of $T$, where $f(v):=d$ for each vertex
 $v$ of $T$. Let $\rmc[f](T)$ be the minimum cardinality of an
 outdegree-$f$ rooted covering of $T$. For a vertex $r$ of an unrooted
 tree $T$, let $\rmc[f](T,r)$ be the minimum cardinality of an
 outdegree-$f$ rooted covering of the rooted tree $(T,r)$. We now show
 that the problem of determining a covering of an unrooted tree can be
 reduced to the case of rooted trees.


\begin{lemma}
  \lemlabel{NonRootedMinCover} Let $f$ be a binding function of a
  (non-rooted) tree $T$. Then $$\mincover[f](T)\;=\;\min_{r\in
    V(T)}\rmc[g](T,r),$$ where $g$ is the binding function of $(T,r)$
  defined by $g(r):=f(r)$ and $g(x):=f(x)-1$ for every vertex $x$ of
  $T-r$.
\end{lemma}

\begin{proof}
  First we prove the lower bound on $\mincover[f](T)$. By
  \lemref{Growing} there is a degree-$f$ covering \C\ of $T$ with some
  vertex $r$ in every subtree of \C, and
  $|\C|=\mincover[f](T)$. Consider a subtree $X\in\C$. Every vertex
  $v\neq r$ in the rooted subtree $(X,r)$ has outdegree at most
  $f(v)-1$ (since the incoming edge incident to $v$ must in $X$). Thus
  $v$ has outdegree at most $g(v)$ in $(X,r)$. The outdegree of $r$ in
  $(X,r)$ equals the degree of $r$ in $X$, which is at most
  $f(r)=g(r)$. Hence \C\ is a rooted $g$-covering of $(T,r)$,
  implying $$\mincover[f](T)=|\C|\geq\min_{r\in
    V(T)}\rmc[g](T,r)\enspace.$$

  Now we prove the upper bound on $\mincover[f](T)$. Let $r$ be a
  vertex in $T$ that minimises $\rmc[g](T,r)$. Thus there is a rooted
  $g$-covering \C\ of $(T,r)$. Consider a rooted subtree $(X,r)$ of
  $\C$. Then $r$ is in $X$, and $\deg_X(r)=\outdeg_{(X,r)}(r)$. For
  every vertex $v$ of $X-r$, we have
  $\deg_X(v)=1+\outdeg_{(X,r)}(v)$. It follows that \C\ is a
  degree-$f$ covering of $T$. Hence
$$\mincover[f](T)\;\leq\;|\C|\;=\;\min_{r\in V(T)}\rmc[g](T,r)\enspace,$$ as desired.
\end{proof}

The next lemma determines $\rmc[f](T)$ precisely.

\begin{lemma}
  \lemlabel{RootedMinCover} Let $f$ be a binding function of a rooted
  tree $(T,r)$.  Let $v_1,\dots,v_{\deg(r)}$ be the neighbours of $r$
  in $T$.  For $i\in\{1,\dots,\deg(r)\}$, let $T_i$ be the component
  subtree of $T-r$ that contains $v_i$.  Let $c_i:=\rmc[f](T_i,v_i)$,
  where $f$ is restricted to $V(T_i)$.  Then $$\rmc[f](T,r)
  =\MAXM{\max_{1\leq
      i\leq\deg(r)}c_i,\CEIL{\frac{1}{f(r)}\sum_{i=1}^{\deg(r)}c_i}}
  \enspace.$$
\end{lemma}

\begin{proof}
  We first prove the upper bound on $\rmc[f](T,r)$. For
  $i\in\{1,\dots,\deg(r)\}$, let $\C_i$ be a degree-$f$ covering of
  $T_i$ with $|\C_i|=c_i$. Let $G$ be the graph with vertex set
  \begin{equation*}
    V(G):=\bigcup_{i=1}^{\deg(r)}\C_i\enspace.
  \end{equation*}
  That is, there is one vertex in $G$ for each subtree in each
  covering $\C_i$. Two vertices in $G$ are adjacent if and only if
  they come from distinct $\C_i$. Thus $G$ is isomorphic to the
  complete $\deg(r)$-partite graph $K\langle
  c_1,\dots,c_{\deg(r)}\rangle$.

  By \lemref{CMG}, there is a partition \F\ of $V(G)$ into
$$\MAXM{\max_{1\leq i\leq\deg(r)}c_i, \CEIL{\frac{1}{f(r)}\sum_{i=1}^{\deg(r)}c_i}}$$ $(\leq f(r))$-cliques in $G$. Each $k$-clique $C\in\F$ corresponds to a set of $k$ subtrees from distinct coverings $\C_i$. Let $X_C$ be the subtree of $T$ induced by the union of the subtrees corresponding to $C$ plus the vertex $r$. Thus $r$ has outdegree $|C|\leq f(r)$ in $X_C$. Since each $\C_i$ is outdegree-$f$, every vertex $x\neq r$ in $X_C$ has outdegree at most $f(x)$ in $X_C$. Thus $\{X_C:C\in\F\}$ is an outdegree-$f$ covering of $(T,r)$. Hence $\rmc[f](T,r)\leq|\F|$, as desired.

We now prove the lower bound on $\rmc[f](T,r)$.  Let \C\ be an
outdegree-$f$ covering of $(T,r)$ with $|\C|=\rmc[f](T,r)$.  Let \X\
be the union, taken over all $X\in\C$, of the set of component
subtrees of $X-r$. Each subtree in \X\ is contained within exactly one
component subtree $T_i$ of $T-r$. For $i\in\{1,\dots,\deg(r)\}$, let
$\X_i$ be the set of subtrees in \X\ that are contained within $T_i$.

We claim that $\X_i$ is an outdegree-$f$ covering of
$(T_i,v_i)$. Every edge of $T_i$ is in some subtree of $\X_i$. For
every vertex $x$ in $T_i$, we have $\outdeg_{T_i}(x)=\outdeg_T(x)$
(since the edge $rv_i$ is incoming at $v_i$). Thus $x$ has outdegree
at most $f(x)$ in every subtree in $\X_i$. Thus $\X_i$ is an
outdegree-$f$ covering of $(T_i,v_i)$. Hence $|\X_i|\geq
\rmc[f](T_i,v_i)=c_i$.

Let $G$ be the graph with vertex set $V(G):=\X$, where two subtrees in
\X\ are adjacent in $G$ if and only if they are in distinct $\X_i$.
Hence $G$ contains the complete $\deg(r)$-partite graph $K\langle
c_1,\dots,c_{\deg(r)}\rangle$ as a subgraph.  For each $X\in\C$,
distinct components of $X-r$ are in distinct components of $T-r$. Thus
the components of $X-r$ are a $k$-clique in $G$, where
$k=\outdeg_X(r)$, which is at most $f(r)$.  Hence \C\ defines a
partition of $V(G)$ into $|\C|$ cliques each with at most $f(r)$
vertices. By \lemref{CMG}, $$\rmc[f](T,r)=|\C|\geq \MAXM{\max_{1\leq
    i\leq\deg(r)}c_i,
  \CEIL{\frac{1}{f(r)}\sum_{i=1}^{\deg(r)}c_i}}\enspace.$$
\end{proof}

\begin{theorem}
  \thmlabel{AlgRootedMinCover} There is a \Oh{n\log n}-time algorithm
  that, given a binding function $f$ of a rooted $n$-vertex tree $T$,
  computes $\rmc[f](T)$.
\end{theorem}

\begin{proof}
  \lemref{RootedMinCover} gives a recursive algorithm to compute
  \rmc[f](T). Let $t(m)$ be the time complexity of this algorithm for
  a tree $T$ with $m$ edges.  We claim that $t(m)\leq\alpha m\log m$
  for some constant $\alpha$. (All logarithms are binary.)\ Let $r$ be
  the root of $T$. Let $v_1,\dots,v_{\deg(r)}$ be the neighbours of
  $r$ in $T$. Let $T_i$ be the component subtree of $T-r$ that
  contains $v_i$. By induction, $\rmc[f](T_i,v_i)$ can be computed in
  $\alpha m_i\log m_i$ time, where $T_i$ has $m_i$ edges. By
  \lemref{RootedMinCover}, $\rmc[f](T)$ can be computed by $2\deg(r)$
  arithmetic steps, each operating on integers at most $m$. This
  computation takes $\alpha\deg(r)\log m$ time. Thus
$$t(m)\leq \alpha\deg(r)\log m+\sum_i \alpha m_i\log m_i \leq \alpha(\log m)(\deg(r)+\sum_i m_i)=\alpha m\log m\enspace.$$ The result follows since $m=n-1$.
\end{proof}


A proof analogous to that of \thmref{AlgRootedMinCover}, but also
using \lemref{CMG}, gives:

\begin{theorem}
  \thmlabel{AlgRootedMinCoverCover} There is a \Oh{n^2}-time algorithm
  that, given a binding function $f$ of a rooted $n$-vertex tree $T$,
  computes $\rmc[f](T)$ and the corresponding covering of $T$ by
  degree-$f$ subtrees.
\end{theorem}

We now have a polynomial time algorithm for the unrooted covering
problem.

\begin{theorem}
  \thmlabel{AlgMinCover} There is a \Oh{n^2\log n}-time algorithm
  that, given a binding function $f$ of an $n$-vertex tree $T$,
  computes $\mincover[f](T)$ and the corresponding covering of $T$ by
  degree-$f$ subtrees.
\end{theorem}

\begin{proof}
  For each vertex $r$ of $T$, compute $\rmc[g](T,r)$, where $g$ is the
  binding function of $(T,r)$ defined by $g(r):=f(r)$ and
  $g(x):=f(x)-1$ for every vertex $x$ of $T-r$. By
  \thmref{AlgRootedMinCover}, $\rmc[g](T,r)$ can be computed to
  \Oh{n\log n} time. Thus this step takes \Oh{n^2\log n} time. Let $r$
  be the vertex that minimises $\rmc[g](T,r)$. This computation takes
  \Oh{n\log n} time. By \lemref{NonRootedMinCover},
  $\mincover[f](T)=\rmc[g](T,r)$.  By \thmref{AlgRootedMinCoverCover},
  the degree-$g$ covering of $(T,r)$ can be computed in \Oh{n^2}
  time. By \lemref{NonRootedMinCover}, this is an optimal degree-$f$
  covering of $T$. The total time complexity is \Oh{n^2\log n}.
\end{proof}




\section{Coverings of Complete Trees}
\seclabel{CompleteTrees}

This section applies the general methods from the previous section to
determine minimum coverings of complete trees. As illustrated in
\figref{CompleteTrees}, for integers $\Delta\geq1$ and $h\geq 1$, the
\emph{complete $\Delta$-ary rooted tree with height $h$}, denoted by
\RootedCompleteTree{\Delta}{h}, is the rooted tree such that every
non-leaf vertex has out-degree $\Delta$, and the distance between the
root and every leaf equals $h$. For convenience, define
$\RootedCompleteTree{\Delta}{0}:=K_1$. For integers $\Delta\geq2$ and
$h\geq 1$, the (non-rooted) \emph{complete $\Delta$-ary tree with
  height $h$}, denoted by \CompleteTree{\Delta}{h}, is the
(non-rooted) tree in which every non-leaf vertex has degree $\Delta$,
and for some vertex $r$, the distance between $r$ and every leaf
equals $h$. Define $\CompleteTree{\Delta}{0}:=K_1$.

\Figure{CompleteTrees} {\includegraphics{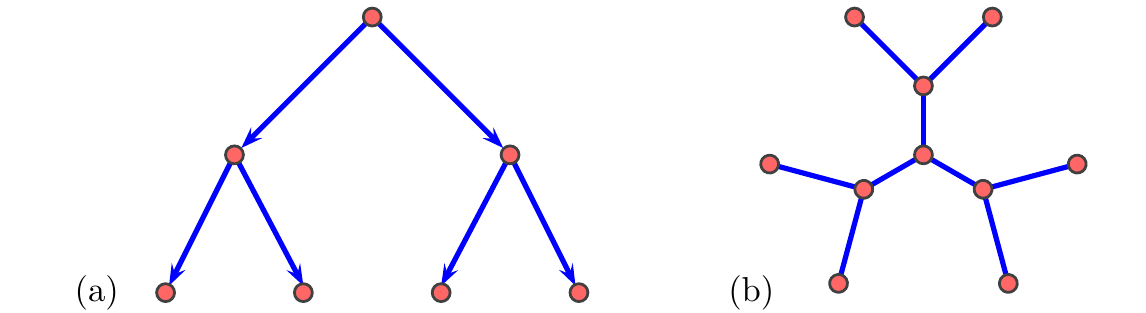}} {(a)
  \RootedCompleteTree{2}{2} and (b) \CompleteTree{3}{2}.}

Consider the following recursively defined function. For every real
number $x>0$, let $\ceill{x}{0}:=1$, and for every integer $k\geq1$,
let
$\ceill{x}{k}:=\ceil{x\cdot\ceill{x}{k-1}}$. Thus $$x^k\leq\ceill{x}{k}\leq\ceil{x}^k,$$
with equality whenever $x$ is an integer. As an example when equality
does not hold, observe that $(\frac{3}{2})^2=\frac{9}{4}$ and
$\Ceill{\frac{3}{2}}{2}=3$ and $\Ceil{\frac{3}{2}}^2=4$.  On the other
hand, $\ceill{x}{k}$ is never far from $x^k$, since $\ceill{x}{k}\leq
x\ceill{x}{k-1}+1$ implies
that $$\ceill{x}{k}\leq\frac{x^{k+1}-1}{x-1}\enspace.$$


\begin{proposition}
  \proplabel{RootedCompleteTree} For all integers $\Delta\geq d\geq 1$
  and
  $h\geq0$, $$\rmc(\RootedCompleteTree{\Delta}{h})=\CEILL{\frac{\Delta}{d}}{h}.$$
\end{proposition}

\begin{proof}
  We proceed by induction on $h$. Trivially,
$$\rmc(\RootedCompleteTree{\Delta}{0})=1=\CEILL{\frac{\Delta}{d}}{0}.$$
Now assume that $h\geq1$. Let $r$ be the root of
\RootedCompleteTree{\Delta}{h}.  Observe that each of the $\Delta$
components of $\RootedCompleteTree{\Delta}{h}-r$ is isomorphic to
\RootedCompleteTree{\Delta}{h-1}, rooted at the neighbour of $r$. By
\lemref{RootedMinCover},
$$\rmc(\RootedCompleteTree{\Delta}{h})=
\MAXM{\rmc(\RootedCompleteTree{\Delta}{h-1}),
  \CEIL{\frac{\Delta\cdot\rmc(\RootedCompleteTree{\Delta}{h-1})}{d}}}
\enspace.$$ By induction and since $\Delta\geq d$,
$$\rmc(\RootedCompleteTree{\Delta}{h})
=\CEIL{\frac{\Delta}{d}\cdot\CEILL{\frac{\Delta}{d}}{h-1}}
=\CEILL{\frac{\Delta}{d}}{h} \enspace,$$ as desired.
\end{proof}

\begin{proposition}
  \proplabel{NonRootedCompleteTree} For all integers $\Delta\geq d\geq
  1$ and $h\geq1$, $$\mincover(\CompleteTree{\Delta}{h})=
  \CEIL{\frac{\Delta}{d}\CEILL{\frac{\Delta-1}{d-1}}{h-1}}.$$
\end{proposition}

\begin{proof}
  By \lemref{NonRootedMinCover},
  \begin{equation}
    \eqnlabel{NonRootedCompleteTree}
    \mincover[d](\CompleteTree{\Delta}{h})\;=\;
    \min_{r\in V(\CompleteTree{\Delta}{h})}\rmc[g](\CompleteTree{\Delta}{h},r),
  \end{equation}
  where $g$ is the binding function of \CompleteTree{\Delta}{h}
  defined by $g(r):=d$ and $g(x):=d-1$ for every vertex $x\neq
  r$. Note that $g$ depends on the choice of $r$.

  \CompleteTree{\Delta}{h} has a vertex $v$ such that each component
  of $\CompleteTree{\Delta}{h}-v$, rooted at the neighbour of $v$, is
  \RootedCompleteTree{\Delta-1}{h-1}. First we compute
  $\rmc[g](\CompleteTree{\Delta}{h},v)$. Later we prove that $v=r$
  minimises $\rmc[g](\CompleteTree{\Delta}{h},r)$ in
  \eqnref{NonRootedCompleteTree}. Each component of
  $\CompleteTree{\Delta}{h}-v$, rooted at the neighbour of $v$, is
  isomorphic to \RootedCompleteTree{\Delta-1}{h-1}. By
  \lemref{RootedMinCover},
  \begin{align*}
    \rmc[g](\CompleteTree{\Delta}{h},v)=
    \MAXM{\rmc[g](\RootedCompleteTree{\Delta-1}{h-1}),
      \CEIL{\frac{\Delta\cdot\rmc[g](\RootedCompleteTree{\Delta-1}{h-1})}{g(v)}}}
    \enspace.
  \end{align*}
  Since $\Delta\geq d=g(v)$ and $g(x)=d-1$ for every vertex $x\neq v$,
  \begin{align*}
    \rmc[g](\CompleteTree{\Delta}{h},v)
    =\CEIL{\frac{\Delta}{d}\cdot\rmc[d-1](\RootedCompleteTree{\Delta-1}{h-1})}
    \enspace.
  \end{align*}

  We now prove that $r:=v$ minimises
  $\rmc[g](\CompleteTree{\Delta}{h},r)$ in
  \eqnref{NonRootedCompleteTree}. Let $w\neq v$ be a vertex in
  \CompleteTree{\Delta}{h}. Then some component of
  $\CompleteTree{\Delta}{h}-w$, rooted at the neighbour of $w$,
  contains \RootedCompleteTree{\Delta-1}{h} rooted at $v$. Thus with
  $g$ defined with respect to $w$,
  \begin{align*}
    \rmc[g](\CompleteTree{\Delta}{h},w)
    \geq\rmc[g](\RootedCompleteTree{\Delta-1}{h})
    &=\rmc[d-1](\RootedCompleteTree{\Delta-1}{h})\\
    &=\CEIL{\frac{\Delta-1}{d-1}\cdot\rmc[d-1](\RootedCompleteTree{\Delta-1}{h-1})}\\
    &\geq\CEIL{\frac{\Delta}{d}\cdot\rmc[d-1](\RootedCompleteTree{\Delta-1}{h-1})}\\
    &=\rmc[g](\CompleteTree{\Delta}{h},v)\enspace.
  \end{align*}
  Hence $r:=v$ minimises $\rmc[g](\CompleteTree{\Delta}{h},r)$ in
  \eqnref{NonRootedCompleteTree}. Thus
$$\mincover[g](\CompleteTree{\Delta}{h})
=\rmc[g](\CompleteTree{\Delta}{h},v)
=\CEIL{\frac{\Delta}{d}\cdot\rmc[d-1](\RootedCompleteTree{\Delta-1}{h-1})}.$$
By \propref{RootedCompleteTree},
$$\rmc[d-1](\RootedCompleteTree{\Delta-1}{h-1})
=\CEILL{\frac{\Delta-1}{d-1}}{h-1}.$$ Thus
\begin{align*}
  \mincover[g](\CompleteTree{\Delta}{h})
  =\CEIL{\frac{\Delta}{d}\CEILL{\frac{\Delta-1}{d-1}}{h-1}}\enspace,
\end{align*}
as desired.
\end{proof}

\section{Coverings of Caterpillars}
\seclabel{Caterpillars}

Consider the problem of covering a given tree with subtrees of bounded
degree.  Since a tree with maximum degree $d$ has at least $d$ leaves,
\thmref{FewLeaves} implies that for every integer $d\geq2$, every tree
$T$ with $\ell$ leaves can be covered by $\Ceil{\frac{\ell}{d}}$
degree-$d$ subtrees. However, the number of leaves can be very large,
even in trees that can be covered by a few subtrees of bounded degree,
as we now show for caterpillars\footnote{A \emph{caterpillar} is a
  tree for which a path is obtained by deleting the leaves.}.

\begin{theorem}
  \thmlabel{Caterpillar} For all integers $\Delta\geq d\geq3$, every
  degree-$\Delta$ caterpillar $T$ has a covering by
  $\Ceil{\frac{\Delta-2}{d-2}}$ degree-$d$ subtrees. Conversely, for
  all integers $\Delta\geq d\geq3$, there are infinitely many
  degree-$\Delta$ caterpillars $T$ such that at least
  $\Ceil{\frac{\Delta-2}{d-2}}$ subtrees are needed in every covering
  of $T$ by degree-$d$ subtrees.
\end{theorem}

\begin{proof} 
  Let $t:=\Ceil{\frac{\Delta-2}{d-2}}$. We first prove that every
  degree-$\Delta$ caterpillar $T$ has a covering by $t$ degree-$d$
  subtrees. Let $L$ be the set of leaves of $T$. Let $P$ be the path
  $T-L$.  Consider a vertex $x$ of $P$. Thus $1\leq\deg_P(x)\leq2$. As
  illustrated in \figref{Caterpillar}, partition the at most
  $\Delta-\deg_P(x)$ leaf edges incident to $x$ into
  $\Ceil{\frac{\Delta-\deg_P(x)}{d-\deg_P(x)}}$ sets each with at most
  $d-\deg_P(x)$ elements. Since $\Ceil{\frac{\Delta-1}{d-1}}\leq t$,
  we have partitioned the leaf edges of $T$ into $t$ sets, such that
  the union of $P$ and any one set is a degree-$d$ subtree of
  $T$. Every edge is in at least one such subtree.

  \Figure{Caterpillar}{\includegraphics{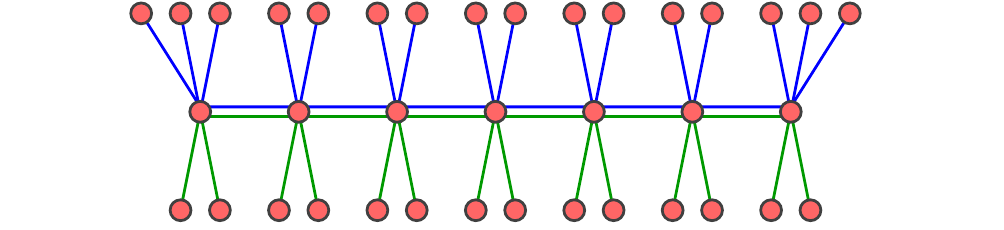}}{Covering a
    degree-$6$ caterpillar by two degree-$4$ subtrees.}

  Now we show that this bound is best possible. For $n\geq 2t-1$, let
  $T_n$ be the caterpillar obtained from the path
  $(u,v_1,\dots,v_{n},w)$ by adding $\Delta-2$ leaves incident to
  $v_i$ for $i\in\{1,\dots,n\}$. Thus each such vertex $v_i$ has
  degree $\Delta$. Every other vertex is a leaf, and $T_n$ is a
  degree-$\Delta$ caterpillar. Suppose on the contrary that $T_n$ can
  be covered by $t-1$ degree-$d$ subtrees $F_1,\dots,F_{t-1}$. Say a
  subtree $F_j$ \emph{hits} a vertex $v_i$ if at least $d-1$ leaf
  edges incident to $v_i$ are in $F_j$. If some vertex $v_i$ is not
  hit, then each subtree contains at most $d-2$ leaf edges incident to
  $v_i$, which is not possible since $(t-1)(d-2)<\Delta-2$. Thus each
  vertex $v_i$ is hit by at least one subtree. Hence the total number
  of hits is at least $n$. Since $n>2(t-1)$, some subtree $F_j$ hits
  at least three vertices, say $v_a,v_b,v_c$ where $1\leq a<b<c\leq
  n$. Since $F_j$ is connected, $F_j$ contains the path
  $(v_a,v_{a+1},\dots,v_c)$. Thus $v_b$ has degree at least $(d-1)+2$
  in $F_j$, which contradicts the assumption that $F_j$ has maximum
  degree at most $d$.
\end{proof}

\section{Pathwidth and Rooted Coverings}
\seclabel{RootedPathwidth}

While \secref{TreeCovering} describes an algorithm for computing
minimal coverings of a given tree by degree-$d$ subtrees, this section
and the next considers the following question: which classes of trees
have coverings by a bounded number of degree-$d$ subtree? The results
in \secref{Caterpillars} say that caterpillars are such a class. To
answer this question more fully, the concept of pathwidth will be
important.  Pathwidth is an important parameter in structural and
algorithmic graph theory, and can be defined in many ways. For forests
we have the following recursive definition, which is easily seen to be
equivalent to the standard definition in terms of path decompositions:
\begin{enumerate}
\item the \emph{pathwidth} of $K_1$ is $0$,
\item the \emph{pathwidth} of a forest $F$ is the maximum pathwidth of
  a connected component of $F$,
\item the \emph{pathwidth} of a tree $T$ is the minimum $k$ such that
  there exists a path $P$ of $T$ and the pathwidth of $T-V(P)$ is at
  most $k-1$.
\end{enumerate}

Caterpillars are precisely the trees of pathwidth $1$. In general, the
pathwidth of an $n$-vertex tree can be computed in \Oh{n} time
\citep{EST-IC94, Skodonis-JAlg03, Bodlaender-SJC96}, and is at most
\Oh{\log n} \citep{Bodlaender-TCS98}. To generalise
\thmref{Caterpillar} for graphs of given pathwidth, we first consider
a rooted variant of the problem in \secref{RootedPathwidth}, and then
we consider the unrooted version in \secref{UnrootedPathwidth}.

We now consider rooted coverings of rooted trees with given pathwidth,
where (for our purposes) the \emph{pathwidth} of a directed graph is
defined to be the pathwidth of the underlying undirected graph. For
all integers $\Delta\geq d\geq 3$ and $k\geq0$, let $\pi(\Delta,d,k)$
be the maximum of $\rmc[d](T)$, where $T$ is an outdegree-$\Delta$
rooted tree with pathwidth $k$. That is, $\pi(\Delta,d,k)$ is the
minimum integer such that every outdegree-$\Delta$ rooted tree with
pathwidth $k$ has an outdegree-$d$ rooted covering with
$\pi(\Delta,d,k)$ subtrees. Below we show that $\pi(\Delta,d,k)$ is
finite. In particular, we prove that $\pi(\Delta,d,k)$ satisfies the
following recurrence, thus determining $\pi(\Delta,d,k)$ precisely.

\begin{theorem}
  \thmlabel{RootedPathwidth} For every integer $\Delta\geq d\geq3$,
$$\pi(\Delta,d,0)=1\enspace,$$ 
and for every integer $k\geq1$, if $\Delta=d$ then
$$\pi(\Delta,d,k)=1\enspace,$$ 
and if $\Delta=d+1$ and $t:=\pi(\Delta,d,k-1)\bmod{d(d-1)}$ then
$$\pi(d+1,d,k)=
\begin{cases}
  \CEIL{\frac{d-1}{d-2}\cdot\pi(\Delta,d,k-1)-\frac{2}{d-2}\FLOORFRAC{\pi(\Delta,d,k-1)}{d(d-1)}}	& \text{if }t<(d-1)(d-2)\\
  \CEIL{\frac{d}{d-1}\cdot\pi(\Delta,d,k-1)+\CEILFRAC{\pi(\Delta,d,k-1)}{d(d-1)}}
  & \text{if }t\geq (d-1)(d-2)\enspace,
\end{cases}$$ and if $\Delta\geq d+2$ then,
$$\pi(\Delta,d,k)=
\CEIL{\frac{\Delta-2}{d}\cdot\pi(\Delta,d,k-1)+
  \frac{2}{d}\CEIL{\frac{\Delta-1}{d-1}\cdot\pi(\Delta,d,k-1)}}
\enspace.$$
\end{theorem}

First observe that $\pi(d,d,k)=1$ since a tree covers itself, and that
$\pi(\Delta,d,0)=1$ since the only tree with pathwidth $0$ is
$K_1$. We now prove the upper bound on $\pi(\Delta,d,k)$, which indeed
shows that $\pi(\Delta,d,k)$ is finite and well defined. The next
lemma will facilitate our inductive proof.

\begin{lemma}
  \lemlabel{RootedPath} For every vertex $r$ of a tree $T$ (with at
  least one edge) there is a degree-$3$ subtree $H$ of $T$ such that:
  \begin{itemize}
  \item the pathwidth of $T-V(H)$ is less than the pathwidth of $T$,
  \item $r$ is in $H$ and $\deg_H(r)\in\{1,2\}$,
  \item there is at most one vertex of $H$ with degree $3$, and
  \item if there is a vertex of $H$ with degree $3$, then
    $\deg_H(r)=1$.
  \end{itemize}
\end{lemma}

\begin{proof} 
  By definition, there is a path $P$ of $T$, such that the pathwidth
  of $T-V(P)$ is less than the pathwidth of $T$. Extend $P$ so that it
  has at least one edge. Let $Q$ be the (possibly empty) path from $r$
  to $P$ in $T$. It is easily verified that $H=P\cup Q$ satisfies the
  lemma.
\end{proof}

\begin{proof}[Proof of Upper Bound in \thmref{RootedPathwidth} for
  $\Delta\geq d+2$] To prove an upper bound on $\pi(\Delta,d,k)$ we
  construct the desired rooted covering of a given outdegree-$\Delta$
  rooted tree with pathwidth $k$. We proceed by induction on $k\geq
  1$. For the base case, suppose that $k=1$. Then $\pi(\Delta,d,0)=1$
  and $t=1<(d-1)(d-2)$. Thus the theorem claims that
  $\pi(\Delta,d,1)=\CEIL{\frac{\Delta-2}{d-2}}$, which is proved in
  \thmref{Caterpillar}. Now assume that $k\geq2$.  Let $T$ be an
  outdegree-$\Delta$ rooted tree with pathwidth $k$. Let $r$ be the
  root of $T$. By \lemref{RootedPath}, there is a degree-$3$ subtree
  $H$ of $T$ such that the pathwidth of $T-V(H)$ is at most $k-1$, $r$
  is in $H$ with degree $1$ or $2$, there is at most one vertex of $H$
  with degree $3$, and if there is a vertex of $H$ with degree $3$,
  then $\deg_H(r)=1$. Thus, if $H$ inherits the orientation of $T$,
  then $H$ has outdegree at most $2$, and at most one vertex in $H$
  has outdegree $2$. As a shorthand, define
$$\pi:=\pi(\Delta,d,k-1)\text{ and }\Lambda:=\CEIL{\frac{\Delta-1}{d-1}\pi}\enspace.$$

For each vertex $v$ in $H$, let $T_v$ be the component of $T-E(H)$
that contains $v$. Thus $T_v$ is rooted at $v$ (in the orientation of
$T$), as illustrated in \figref{RootedPathwidth}.

\Figure{RootedPathwidth}{\includegraphics{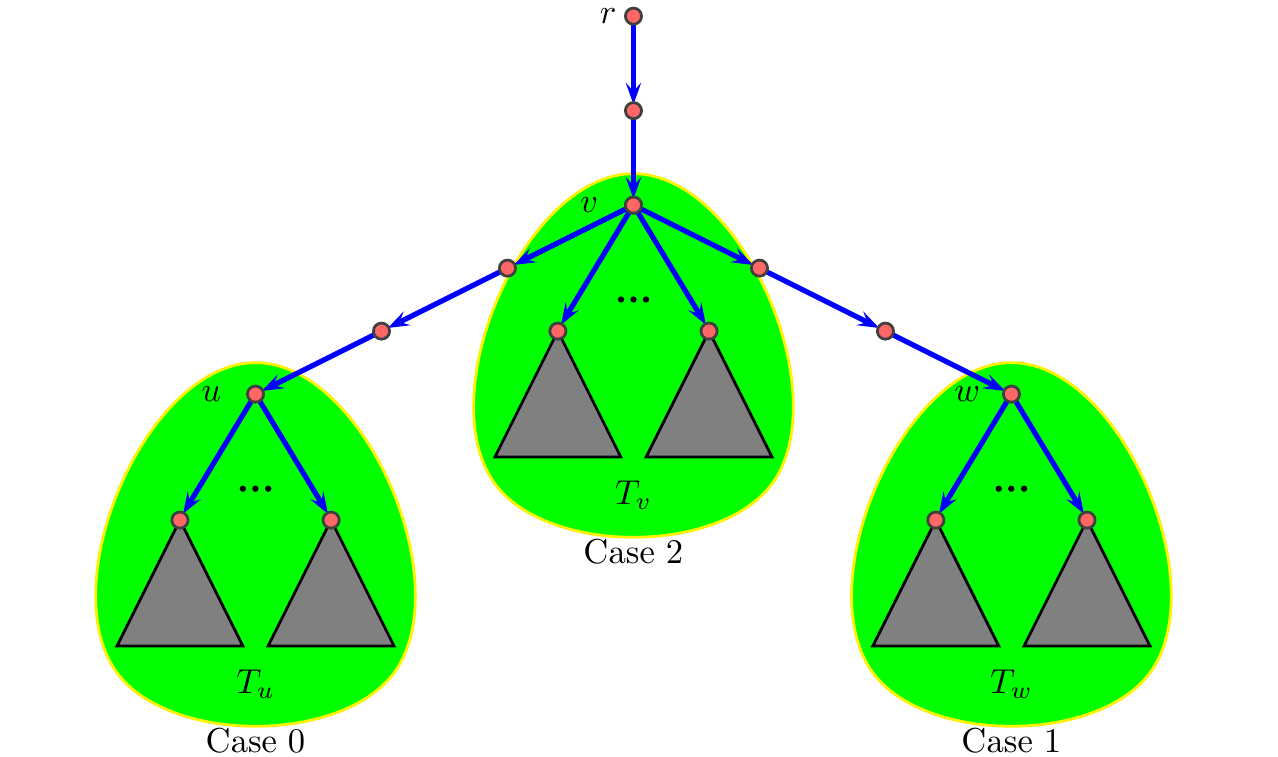}}{Construction
  in the proof of the upper bound in \thmref{RootedPathwidth}.}

We now determine a rooted covering of $T_v$. For each neighbour $w$ of
$v$ in $T_v$, the component of $T_v-v$ that contains $w$ has pathwidth
at most $k-1$. Thus by induction this component (rooted at $w$) has a
rooted covering by $\pi$ subtrees.

\textbf{\boldmath Case 0. $\outdeg_H(v)=0$:} Let $f$ be the binding
function of $T_v$ defined by $f(x):=d$ for every vertex $x$ in
$T_v$. By \lemref{RootedMinCover}, $T_v$ has an outdegree-$d$ rooted
covering $\C_v$, where $$|\C_v|\leq\CEIL{\frac{\Delta}{d}\pi}\leq
\Lambda\enspace.$$

\textbf{\boldmath Case 1. $\outdeg_H(v)=1$:} Let $f$ be the binding
function of $T_v$ defined by $f(v):=d-1$ and $f(x):=d$ for every
vertex $x$ in $T_v-v$. Since $v$ has outdegree at most $\Delta-1$ in
$T_v$, by \lemref{RootedMinCover} applied to $f$, $T_v$ has an
outdegree-$d$ rooted covering $\C_v$, such that $v$ has outdegree at
most $d-1$ in each subtree in $\C_v$, and
$$|\C_v|\leq \CEIL{\frac{\Delta-1}{d-1}\pi}=\Lambda\enspace.$$

\textbf{\boldmath Case 2. $\outdeg_H(v)=2$:} By induction, each
component of $T_v-v$, rooted at the neighbour of $v$, has an
outdegree-$d$ rooted covering consisting of $\pi$ subtrees. Let $G$ be
the graph with one vertex for each subtree in the coverings of the
components of $T_v-v$, where two vertices are adjacent if the
corresponding subtrees come from distinct components. Since $v$ has
outdegree at most $\Delta-2$ in $T_v$, $G$ is isomorphic to a subgraph
of the Turan $(\Delta-2)$-partite graph with $\pi$ vertices in each
colour class. Apply \corref{CMGCMG} with $n=(\Delta-2)\pi$ and $p=d-2$
and $q=d$ and $m=\Lambda$. (Observe that $\Delta\geq d+2$ implies that
$\Delta-2\geq p,q$, and thus, \corref{CMGCMG} is applicable.)\ Hence
there is a partition of $V(G)$ into $\Lambda$ $(d-2)$-cliques and
$$\CEIL{\frac{\max\{(\Delta-2)\pi-(d-2)\Lambda,0\}}{d}}$$
$(\leq d)$-cliques. Since $\Delta\geq d$ we have $(d-1)(\Delta-2)\pi
\geq (d-2)(\Delta-1)\pi$.  Thus $(d-1)(\Delta-2)\pi + d(d-1)
>(d-2)(\Delta-1)\pi+(d-1)(d-2)$. That is, $(\Delta-2)\pi + d >
(d-2)(\frac{\Delta-1}{d-1}\pi+1)>(d-2)\Lambda$.  Hence $(\Delta-2)\pi
-(d-2)\Lambda > -d$, implying $\ceil{\frac1d\big((\Delta-2)\pi
  -(d-2)\Lambda\big)}\geq 0$.  Hence
$\ceil{\frac1d\max\{(\Delta-2)\pi-(d-2)\Lambda,0\}}=
\ceil{\frac1d\big((\Delta-2)\pi-(d-2)\Lambda\big)}$.  Therefore there
is a partition of $V(G)$ into $\Lambda$ $(d-2)$-cliques and
$$\CEIL{\frac{(\Delta-2)\pi-(d-2)\Lambda}{d}}$$
$(\leq d)$-cliques. Using the method in \lemref{RootedMinCover}, it
follows that $T_v$ has an outdegree-$d$ rooted covering
$\C_v\cup\D_v$, such that $|\C_v|\leq \Lambda$ and $v$ has outdegree
$d-2$ in each subtree in $\C_v$; and
$$|\D_v|\leq\CEIL{\frac{(\Delta-2)\pi-(d-2)\Lambda}{d}}\enspace,$$
and $v$ has outdegree at most $d$ in each subtree in $\D_v$.

\smallskip Note that $|\C_v|\leq \Lambda$ for every vertex $v$ of $H$.
For $i\in\{1,\dots,\Lambda\}$, let $X_i$ be the union, taken over
every vertex $v$ in $H$, of the $i$-th subtree in $\C_v$. Observe that
in the construction in Case ($j$), $\outdeg_H(v)=j$ and $v$ has
outdegree at most $d-j$ in each subtree in $\C_v$. Thus $X_i\cup H$
has outdegree at most $d$, and $X_i\cup H$ contains $r$.

Suppose that $\outdeg_H(v)=2$ for some vertex $v$ in $H$. Let $Q$ be
the directed path from $r$ to $v$ in $H$. (Note that it is possible
that $v=r$.)\ Then for every subtree $Y\in\D_v$, $Y\cup Q$ has
outdegree at most $d$ (since no outgoing edges incident to $v$ are in
$Q$, and every other vertex in $Q$ has outdegree $1$ in $Y\cup Q$, and
$1<d$.)\

Observe that every edge of $T$ is in some subtree $X_i\cup H$ (where
$1\leq i\leq \Lambda$) or some subtree $Y\cup Q$ (where $Y\in\D_v$ and
$\outdeg_H(v)=2$). Hence $$\SET{X_i\cup H:1\leq i\leq
  \Lambda}\cup\SET{Y\cup Q:Y\in D_v}$$ is an outdegree-$d$ rooted
covering of $T$. Therefore
\begin{align*}
  \pi(\Delta,d,k) \leq \Lambda +
  \CEIL{\frac{(\Delta-2)\pi-(d-2)\Lambda}{d}} =
  \CEIL{\frac{2\Lambda+(\Delta-2)\pi}{d}}\enspace,
\end{align*}
as desired.
\end{proof}


\begin{proof}[Proof of Upper Bound in \thmref{RootedPathwidth} for
  $\Delta=d+1$] We proceed by induction on $k\geq 1$. Let
  $\pi:=\pi(\Delta,d,k-1)$. For the base case, suppose that
  $k=1$. Then $\pi=1$ and $t=1<(d-1)(d-2)$. Thus the theorem claims
  that $\pi(d+1,d,1)=\CEIL{\frac{\Delta-2}{d-2}}$, which is proved in
  \thmref{Caterpillar}. Now assume that $k\geq2$. Let $T$ be an
  outdegree-$\Delta$ rooted tree with pathwidth $k$. Let $r$ be the
  root of $T$. By definition, there is an (undirected) path $P$ in
  $T$, such that the pathwidth of $T-V(P)$ is less than $k$. Let $Q$
  be the shortest path in $T$ from $r$ to a vertex in $P$. Let
  $H:=P\cup Q$. Let $s$ be the vertex in $P\cap Q$. Note that it is
  possible that $r\in P$, in which case $s=r$. Let $P_1$ and $P_2$ be
  the subpaths of $P$ such that $P_1\cap P_2=\{s\}$ and $P_1\cup
  P_2=P$. Each $P_i$ is a directed path starting at $s$. Let
  $Q_i:=Q\cup P_i$ for $i\in\{1,2\}$. Each $Q_i$ is a directed path
  starting at $r$.

  For each vertex $v$ in $H$, let $T_v$ be the component of $T-E(H)$
  that contains $v$. Thus $T_v$ is rooted at $v$ (in the orientation
  of $T$). We now determine a rooted covering of $T_v$. For each
  neighbour $w$ of $v$ in $T_v$, the component of $T_v-v$ that
  contains $w$ has pathwidth at most $k-1$. Thus by induction this
  component (rooted at $w$) has a rooted covering by $\pi$ subtrees.

  Let $t:=\pi\bmod{d(d-1)}$. Define
$$y:=
\begin{cases}
  \FLOORFRAC{\pi}{d(d-1)}	& \text{if }t<(d-1)(d-2)\\
  \CEILFRAC{\pi}{d(d-1)} & \text{if }t\geq(d-1)(d-2)\enspace,
\end{cases}$$ and
$$x:=
\begin{cases}
  \CEIL{\frac{d-1}{d-2}\pi-\frac{2d-2}{d-2}y}
  & \text{if }t<(d-1)(d-2)\\
  \CEIL{\frac{d}{d-1}\pi-y} & \text{if }t\geq(d-1)(d-2)\enspace.
\end{cases}$$

It is easily verified that $(d-2)x+2(d-1)y\geq(\Delta-2)\pi$ and
$(d-1)(x+y)\geq(\Delta-1)\pi$ and $d(x+y)>\Delta\pi$ and
$(d-1)\pi+d-2>2y(d-1)$.


\textbf{\boldmath Case 0. $\outdeg_H(v)=0$:} Then $v\in Q_i$ for some
$i\in\{1,2\}$, and $v$ has outdegree at most $\Delta$ in $T_v$. By
\lemref{RootedMinCover}, $T_v$ has an outdegree-$d$ rooted covering
$\C_v\cup\D^i_v$, where
$$|\C_v\cup\D^i_v|\leq\CEIL{\frac{\Delta}{d}\pi}\leq x+y\enspace.$$
Hence we can choose $\C_v$ and $\D^i_v$ so that $|\C_v|\leq x$ and
$|\D^i_v|\leq y$.

\textbf{\boldmath Case 1. $\outdeg_H(v)=1$:} Then $v\in Q_i$ for some
$i\in\{1,2\}$, and $v$ has outdegree at most $\Delta-1$ in $T_v$. By
\lemref{RootedMinCover}, $T_v$ has an outdegree-$d$ rooted covering
$\C_v\cup\D^i_v$, such that $v$ has outdegree at most $d-1$ in each
subtree in $\C_v\cup\D^i_v$, and
$$|\C_v\cup\D^i_v|\leq\CEIL{\frac{\Delta-1}{d-1}\pi}\leq x+y\enspace.$$
Hence we can choose $\C_v$ and $\D^i_v$ so that $|\C_v|\leq x$ and
$|\D^i_v|\leq y$.

\textbf{\boldmath Case 2. $\outdeg_H(s)=2$:} By induction, each
component of $T_s-s$, rooted at the neighbour of $s$, has an
outdegree-$d$ rooted covering consisting of $\pi$ subtrees. Let $G$ be
the graph with one vertex for each subtree in the coverings of the
components of $T_s-s$, where two vertices are adjacent if the
corresponding subtrees come from distinct components. Since $s$ has
outdegree at most $\Delta-2$ in $T_s$, $G$ is isomorphic to a subgraph
of the Turan $(\Delta-2)$-partite graph with $\pi$ vertices in each
colour class.

Apply \corref{CMGCMG} with $n=(\Delta-2)\pi=(d-1)\pi$ and $p=d-1$ and
$q=d-2$ and $m=2y$ . (Observe that $\Delta=d+1$ implies that
$\Delta-2\geq p,q$, and thus, \corref{CMGCMG} is applicable.)\ Thus
there is a partition of $V(G)$ into $2y$ $(\leq d-1)$-cliques
and $$\CEIL{\frac{\max\{(d-1)\pi-(d-1)2y,0\}}{d-2}}$$ $(\leq
d-2)$-cliques. Since $(d-1)\pi-2y(d-1)>2-d$, we have
$\frac{(d-1)\pi-2y(d-1)}{d-2}>-1$, implying
$$\CEIL{\frac{(d-1)\pi-2y(d-1)}{d-2}}\geq0\enspace.$$
Thus
$$\CEIL{\frac{\max\{(d-1)\pi-(d-1)2y,0\}}{d-2}}=
\CEIL{\frac{(d-1)\pi-(d-1)2y}{d-2}}\leq x\enspace.$$ Hence there is a
partition of $V(G)$ into $2y$ $(\leq d-1)$-cliques and $x$ $(\leq
d-2)$-cliques. Using the method in \lemref{RootedMinCover}, it follows
that $T_s$ has an outdegree-$d$ rooted covering
$\C_s\cup\D^1_s\cup\D^2_s$, such that $|\D^1_s|\leq y$ and $s$ has
outdegree at most $d-1$ in each subtree in $\D^1_s$; and $|\D^2_s|\leq
y$ and $s$ has outdegree at most $d-1$ in each subtree in $\D^2_s$;
and $|\C_s|\leq x$ and $s$ has outdegree at most $d-2$ in each subtree
in $\C_s$.

\smallskip Observe that $|\C_v|\leq x$ for every vertex $v$ of
$H$. For $j\in\{1,\dots,x\}$, let $C_j$ be the union, taken over every
vertex $v$ in $H$, of the $j$-th subtree in $\C_v$. Observe that in
the construction in Case ($\ell$), $\outdeg_H(v)=\ell$ and $v$ has
outdegree at most $d-\ell$ in each subtree in $\C_v$. Thus $C_j\cup H$
has outdegree at most $d$. By construction, $C_j\cup H$ is connected
and contains $r$.

For $i\in\{1,2\}$, observe that $|\D^i_v|\leq y$ for every vertex $v$
of $Q_i$. For $j\in\{1,\dots,y\}$, let $D^i_j$ be the union, taken
over every vertex $v$ in $Q_i$, of the $j$-th subtree in
$\D^i_v$. Observe that for $\ell\in\{0,1\}$, in the construction in
Case ($\ell$), $\outdeg_{Q_i}(v)=\ell$ and $v$ has outdegree at most
$d-\ell$ in each subtree in $\D^i_v$. In the construction in Case 2,
$\outdeg_{Q_i}(s)=1$ and $v$ has outdegree at most $d-1$ in each
subtree in $\D^i_v$. Thus $D^i_j\cup Q_i$ has outdegree at most
$d$. By construction, $D_j^i\cup Q_i$ is connected and contains $r$.

Observe that every edge of $T$ is in some subtree $C_j\cup H$ (where
$1\leq j\leq x$) or some subtree $D^i_j\cup Q_i$ (where $i\in\{1,2\}$
and $1\leq j\leq y$). Hence $$\SET{C_j\cup H:1\leq j\leq
  x}\cup\SET{D^i_j\cup Q_i:i\in\{1,2\},1\leq j\leq y}$$ is an
outdegree-$d$ rooted covering of $T$. Therefore
\begin{align*}
  \pi(\Delta,d,k) & \leq x+2y=
  \begin{cases}
    \CEIL{\frac{d-1}{d-2}\pi-\frac{2}{d-2}\FLOORFRAC{\pi}{d(d-1)}}
    & \text{if }t<(d-1)(d-2)\\
    \CEIL{\frac{d}{d-1}\pi+\CEILFRAC{\pi}{d(d-1)}}
    & \text{if }t\geq(d-1)(d-2)\enspace,\\
  \end{cases}
\end{align*}
as desired.
\end{proof}

\begin{proof}[Proof of Lower Bound in \thmref{RootedPathwidth}]
  To prove a lower bound on $\pi(\Delta,d,k)$ we construct an
  outdegree-$\Delta$ rooted tree with pathwidth $k$ that requires many
  subtrees in every outdegree-$d$ covering. For all integers
  $n_1,\dots,n_k$, where each $n_i\geq\pi(\Delta,d,i)+1$, we construct
  a tree $T\ang{n_1,\dots,n_k}$ with the desired property. (This
  statement is well-defined since we have already proved that
  $\pi(\Delta,d,i)$ is bounded from above.)\ The number of vertices in
  $T\ang{n_1,\dots,n_k}$ increases with the $n_i$. Hence there are, in
  fact, infinitely many such trees. Each tree has a nominated root
  vertex $r$, which has out-degree $\Delta$. Every non-leaf vertex has
  outdegree $\Delta$ or $\Delta-1$.

  The tree $T\ang{n_1,\dots,n_k}$ is constructed recursively as
  follows, starting from the
  path $$P:=(v_{-n},\dots,v_{-1},v_0,v_1,\dots,v_n)\enspace,$$where
  $n:=n_k$. If $k=1$ then, add $\Delta-2$ leaf vertices adjacent to
  $v_0$, and for each vertex $v_i$ in $P$ with $1\leq |i|\leq n$, add
  $\Delta-1$ leaf vertices adjacent to $v_i$. If $k\geq2$ then,
  connect $v_0$ to the root vertex in each of $\Delta-2$ copies of
  $T\langle n_1,\dots,n_{k-1}\rangle$, and for each vertex $v_i$ in
  $P$ with $1\leq |i|\leq n$, connect $v_i$ to the root vertex in each
  of $\Delta-1$ copies of $T\langle n_1,\dots,n_{k-1}\rangle$. Root
  $T\ang{n_1,\dots,n_k}$ at $r:=v_0$. Thus $v_{-n}$ and $v_n$ have
  outdegree $\Delta-1$, and every other vertex $v_i$ has outdegree
  $\Delta$. By construction, $T\ang{n_1,\dots,n_k}$ has pathwidth $k$.


  By the definition of $\pi$, $T\ang{n_1,\dots,n_k}$ has an
  outdegree-$d$ rooted covering
  $\C=\{F_1,\dots,F_{\pi(\Delta,d,k)}\}$, and $r$ is in each $F_i$.
  We classify these subtrees depending on which edges in $F_i\cap P$
  are incident to $r$.  Let $\C^{++}$ be the set of subtrees
  $F_i\in\C$ such that the edges $rv_{-1}$ and $rv_1$ are both in
  $F_i\cap P$.  Let $\C^{+-}$ be the set of subtrees $F_i\in\C$ such
  that the edge $rv_{-1}$ is the only edge incident to $r$ in $F_i\cap
  P$.  Let $\C^{-+}$ be the set of subtrees $F_i\in\C$ such that the
  edge $rv_{1}$ is the only edge incident to $r$ in $F_i\cap P$.  Let
  $\C^{--}$ be the set of subtrees $F_i\in\C$ such that the edges
  $rv_{-1}$ and $rv_1$ are both not in $F_i\cap P$.  Hence
  \begin{equation}
    \eqnlabel{CCCC}
    |\C^{++}|+|\C^{+-}|+|\C^{-+}|+|\C^{--}|\;=\;\pi(\Delta,d,k)\enspace.
  \end{equation}

  For each subtree $F_i\in\C$, let $F_{i,1},\dots,F_{i,s_i}$ be the
  component subtrees of $F_i-V(P)$. Let $$\F:=\{F_{i,j}:1\leq
  i\leq\pi(\Delta,d,k),1\leq j\leq s_i\}\enspace.$$ Each subtree
  $F_{i,j}\in\F$ is contained in exactly one copy of
  $T\ang{n_1,\dots,n_{k-1}}$ in $T\ang{n_1,\dots,n_k}$. Consider a
  copy $T'$ of $T\ang{n_1,\dots,n_{k-1}}$.  Say $x$ is the root of
  $T'$, and $v_\ell$ is the neighbour of $x$ in $P$.  We say that
  $v_{\ell}$ is the \emph{attachment point} of $T'$ and of each
  subtree $F_{i,j}\in \F$ that is contained in $T'$.  Since $r$ is in
  $F_i$, the path between every vertex in $F_{i,j}$ and $r$ is in
  $F_i$.  This path includes $x$, which is thus in each $F_{i,j}$.
  Since $F_i$ has outdegree at most $d$, each $F_{i,j}$ has outdegree
  at most $d$.  Thus the set of subtrees in \F\ that are contained in
  $T'$ form an outdegree-$d$ rooted covering of $T'$.  Let
  $\pi:=\pi(\Delta,d,k-1) $. By induction, at least $\pi$ subtrees in
  \F\ are contained in $T'$.

  Now partition the subtrees in \F\ according to their attachment
  point in $P$.  Let $\F^{0}$ be the set of subtrees in \F\ whose
  attachment point is $v_0$.  Let $\F^{+}$ be the set of subtrees in
  \F\ whose attachment point is $v_i$ for some $i\in\{1,\dots,n\}$.
  Let $\F^{-}$ be the set of subtrees in \F\ whose attachment point is
  $v_i$ for some $i\in\{-n,\dots,-1\}$.

  There are $\Delta-2$ copies of $T\ang{n_1,\dots,n_{k-1}}$ that
  attach at $v_0$, each of which contain at least $\pi$ subtrees in
  \F. Thus $|\F^{0}|\geq\pi\cdot(\Delta-2)$. For each $F_i\in\C^{++}$,
  since $v_0$ has outdegree $2$ in $F_i\cap P$ and outdegree at most
  $d$ in $F_i$, there are at most $d-2$ component subtrees of $F_i-P$
  that are in $\F^{0}$. Similarly, for each
  $F_i\in\C^{-+}\cup\C^{+-}$, since $v_0$ has outdegree $1$ in
  $F_i\cap P$, there are at most $d-1$ component subtrees of $F_i-P$
  that are in $\F^{0}$. Finally, for each $F_i\in\C^{--}$, there are
  at most $d$ component subtrees of $F_i-P$ that are in
  $\F^{0}$. Hence
  \begin{equation}
    \eqnlabel{middle}
    \pi\cdot(\Delta-2)\leq |\F^{0}|
    \leq (d-2)\cdot|\C^{++}|
    +(d-1)\cdot\big(|\C^{-+}|+|\C^{+-}|\big)
    +d\cdot|\C^{--}|\enspace.
  \end{equation}

  There are $\Delta-1$ copies of $T\ang{n_1,\dots,n_{k-1}}$ that
  attach at $v_i$ for $i\in\{1,\dots,n\}$, each of which contain at
  least $\pi$ subtrees in \F. Thus
  \begin{equation}
    \eqnlabel{Fplus}
    |\F^{+}|\geq\pi\cdot(\Delta-1)n\enspace.
  \end{equation}
  For each $F_i\in\C^{++}\cup\C^{-+}$, the subtree consisting of those
  edges in $F_i$ whose source endpoint is in $\{v_1,\dots,v_n\}$ is an
  outdegree-$d$ caterpillar rooted at $v_1$ whose spine is contained
  in $\{v_1,\dots,v_n\}$. Every outdegree-$d$ caterpillar rooted at
  the endpoint of its spine and whose spine has at most $n$ vertices
  has at most $(d-1)n+1$ leaves. Thus there are at most $(d-1)n+1$
  component subtrees of $F_i-P$ that are in $F^+$. For each
  $F_i\in\C^{--}\cup\C^{+-}$, no component subtrees of $F_i-P$ are in
  $\F^+$. Thus by \eqnref{Fplus},
  \begin{equation*}
    \pi\cdot(\Delta-1)n \leq |\F^{+}|
    \leq \big((d-1)n+1\big)\cdot\big(|\C^{++}|+|\C^{-+}|\big)\enspace.
  \end{equation*}
  Hence
  \begin{align*}
    \pi\cdot(\Delta-1) -(d-1)\cdot\big(|\C^{++}|+|\C^{-+}|\big) & \leq
    \frac{1}{n}\cdot\big(|\C^{++}|+|\C^{-+}|\big)\enspace.
  \end{align*}
  By \eqnref{CCCC} and since $n=n_k>\pi(\Delta,d,k)$,
  \begin{align*}
    \pi\cdot(\Delta-1)-(d-1)\cdot\big(|\C^{++}|+|\C^{-+}|\big) \leq
    \FLOOR{\frac{\pi(\Delta,d,k)}{n}}=0\enspace.
  \end{align*}
  Thus
  \begin{equation}
    \eqnlabel{right}
    \pi\cdot(\Delta-1) \leq (d-1)\cdot\big(|\C^{++}|+|\C^{-+}|\big)\enspace.
  \end{equation}
  By symmetry,
  \begin{equation}
    \eqnlabel{left}
    \pi\cdot(\Delta-1) \leq (d-1)\cdot\big(|\C^{++}|+|\C^{+-}|\big)\enspace.
  \end{equation}
  Observe that \threeeqnref{middle}{right}{left} define an integer
  linear program with unknowns
  $|\C^{++}|$, $|\C^{+-}|$, $|\C^{-+}|$, $|\C^{--}|$. The solution of this
  integer linear program is given in \lemref{IntegerProgram}, where
  \begin{align*}
    x&=|\C^{++}|,\; y_1=|\C^{+-}|,\; y_2=|\C^{-+}|,\; z=|\C^{--}|,\\
    A&=\pi\cdot(\Delta-2),\text{ and }\\
    B&=\pi\cdot(\Delta-1)\enspace.
  \end{align*}
  Since $\Delta\geq d$ we have $(d-2)B\leq(d-1)A$, and
  \lemref{IntegerProgram} is applicable. \Eqnref{CCCC} and
  \lemref{IntegerProgram} imply that
  \begin{align*}
    \pi(\Delta,d,k)&=|\C^{++}|+|\C^{+-}|+|\C^{-+}|+|\C^{--}|\\
    &\geq\CEIL{\frac{\Delta-2}{d}\cdot\pi
      +\frac{2}{d}\CEIL{\frac{\Delta-1}{d-1}\cdot\pi}}\enspace.
  \end{align*}
  This complete the proof of the lower bound when $\Delta\geq d+2$.

  For $\Delta=d+1$ the above analysis can be slightly improved as
  follows. Observe that for each $F_i\in\C^{--}$, there are at most
  $d-1$ component subtrees of $F_i-P$ that are in $\F^{0}$ (rather
  than $d$ component subtrees in the general case). Hence
  \eqnref{middle} can be strengthened to:
  \begin{equation}
    \eqnlabel{MIDDLE}
    \pi\cdot(\Delta-2)\leq |\F^{0}|
    \leq (d-2)\cdot|\C^{++}|
    +(d-1)\cdot\big(|\C^{-+}|+|\C^{+-}|+|\C^{--}|\big)\enspace.
  \end{equation}
  Now consider the integer linear program with unknowns
  $|\C^{++}|,|\C^{+-}|,|\C^{-+}|,|\C^{--}|$ that is defined in
  \threeeqnref{right}{left}{MIDDLE}. The solution of this integer
  linear program is given in \lemref{AnotherIntegerProgram}, where
  \begin{align*}
    x=|\C^{++}|,\; y_1=|\C^{+-}|,\; y_2=|\C^{-+}|,\;
    z=|\C^{--}|,\text{ and } A=\pi\enspace.
  \end{align*}
  \Eqnref{CCCC} and \lemref{AnotherIntegerProgram} imply that
  \begin{align*}
    \pi(\Delta,d,k)&=|\C^{++}|+|\C^{+-}|+|\C^{-+}|+|\C^{--}|\\
    &\geq
    \begin{cases}
      \CEIL{\frac{d-1}{d-2}\cdot\pi-\frac{2}{d-2}\FLOORFRAC{\pi}{d(d-1)}}
      & \text{if }t<(d-1)(d-2)\\
      \CEIL{\frac{d}{d-1}\cdot\pi+\CEILFRAC{\pi}{d(d-1)}}
      & \text{if }t\geq(d-1)(d-2)\enspace.\\
    \end{cases}
  \end{align*}
  as desired in the case that $\Delta\geq d+1$.
\end{proof}

This completes the proof of \thmref{RootedPathwidth}. We can estimate
the recurrence in \thmref{RootedPathwidth} as follows.
  
\begin{corollary}
  For all integers $\Delta\geq d\geq 2$ and $k\geq0$,
$$\pi(\Delta,d,k)  \leq 
\CEIL{\frac{\Delta-2}{d}+\frac{2}{d}\CEIL{\frac{\Delta-1}{d-1}}}^k\enspace,$$
with equality whenever $\Delta\equiv d^2-2d+2\pmod{d^2-d}$.
\end{corollary}

\begin{proof}
  It is easily proved that $\Ceil{\frac{ab}{c}}\leq
  a\Ceil{\frac{b}{c}}$ for all integers $a,b,c\geq1$. Applying this
  observation twice, \thmref{RootedPathwidth} implies that
  \begin{align*}
    \pi(\Delta,d,k) \leq \pi(\Delta,d,k-1)\cdot
    \CEIL{\frac{\Delta-2}{d}+\frac{2}{d}\CEIL{\frac{\Delta-1}{d-1}}}\enspace.
  \end{align*}
  Since $\pi(\Delta,d,0)=1$,
  \begin{equation}
    \eqnlabel{RootedPathwidthEstimate}
    \pi(\Delta,d,k)  \leq 
    \CEIL{\frac{\Delta-2}{d}+\frac{2}{d}\CEIL{\frac{\Delta-1}{d-1}}}^k\enspace.
  \end{equation}
  Now assume that $\Delta\equiv d^2-2d+2\pmod{d^2-d}$. Then
  $\frac{\Delta-1}{d-1}\in\mathbb{Z}$ and
  $\frac{\Delta-2}{d}\in\mathbb{Z}$. (In fact, the converse holds.)\
  Thus \thmref{RootedPathwidth} implies that
  \begin{align*}
    \pi(\Delta,d,k) & =\pi(\Delta,d,k-1)\cdot
    \BRACKET{\frac{\Delta-2}{d}+\frac{2}{d}\cdot \frac{\Delta-1}{d-1}}
    \enspace.
  \end{align*}
  Thus equality in \eqnref{RootedPathwidthEstimate} holds since
  $\pi(\Delta,d,0)=1$.
\end{proof}

\section{Pathwidth and Unrooted Coverings}
\seclabel{UnrootedPathwidth}

This section extends the results in \secref{RootedPathwidth} to the
unrooted setting.

\begin{theorem}
  \thmlabel{Pathwidth} For all integers $\Delta\geq d\geq3$, every
  degree-$\Delta$ tree $T$ with pathwidth $k$ satisfies
  $\mincover[d](T)\leq t$,
  where $$t:=\CEIL{\frac{\Delta-2}{d-2}\cdot\pi}\;\;\;\text{and}\;\;\;\pi:=\pi(\Delta-1,d-1,k-1)\enspace.$$
  Moreover, there are infinitely many degree-$\Delta$ trees $T$ with
  pathwidth $k$ such that $\mincover[d](T)=t$.
\end{theorem}

\begin{proof} First we prove the upper bound.  $T$ has a path $P$ such
  that $T-V(P)$ has pathwidth $k-1$. Consider a vertex $v$ of $P$. Let
  $T_v$ be the subtree of $T-E(P)$ that contains $v$, where $T_v$ is
  rooted at $v$. Let $f$ be the binding function of $T_v$ defined by
  $f(v):=d-\deg_P(v)$ and $f(x):=d-1$ for every other vertex $x$. Each
  component $U$ of $T_v-v$, rooted at the neighbour of $v$, has
  outdegree at most $\Delta-1$ and pathwidth at most $k-1$. Thus
  $\rmc[d-1](U)\leq\pi$. Since $v$ has outdegree at most
  $\Delta-\deg_P(v)$ in $T_v$, by \lemref{RootedMinCover} applied to
  $f$, $T_v$ has an outdegree-$(d-1)$ rooted covering $\C_v$, such
  that $v$ has outdegree at most $d-\deg_P(v)$ in each subtree in
  $\C_v$, and
$$|\C_v|\leq\CEIL{\frac{\Delta-\deg_P(v)}{d-\deg_P(v)}\cdot\pi}
\leq t\enspace,$$ where the last inequality holds since
$\deg_P(v)\in\{0,1,2\}$. For $i\in\{1,\dots,t\}$, let $X_i$ be the
union, taken over every vertex $v$ in $P$, of the $i$-th subtree in
$\C_v$ (if it exists). Thus every vertex $v$ in $P$ has degree at most
$d-\deg_P(v)$ in $X_i$, and $v$ has degree at most $d$ in $X_i\cup
P$. Since every vertex not in $P$ has outdegree at most $d-1$ in each
$X_i$, every vertex not in $P$ has degree at most $d$ in each
$X_i$. Every edge of $T$ is in some $X_i\cup P$. Hence $\SET{X_i\cup
  P:1\leq i\leq t}$ is the desired degree-$d$ covering of $T$.

Now we prove the lower bound. Let $X$ be the outdegree-$(\Delta-1)$
rooted tree with pathwidth $k-1$, such that $\rmc[d-1](X)=\pi$. (See
the proof of the lower bound in \thmref{RootedPathwidth} for the
construction of $X$.)\ Let $n\geq t-1$. Let $T$ be the tree obtained
from the path $P=(v_{-n-1},v_{-n},\dots,v_{n},v_{n+1})$ as
follows. For each $i\in\{-n,\dots,n\}$, add $\Delta-2$ copies of $X$
whose roots are adjacent to $v_i$; thus $v_i$ has degree
$\Delta$. Hence $T$ is a degree-$\Delta$ tree with pathwidth $k$.

Suppose on the contrary that $T$ can be covered by $t-1$ degree-$d$
subtrees.  By \lemref{Growing}, $T$ has a degree-$d$ covering by $t-1$
degree-$d$ maximal subtrees $F_1,\dots,F_{t-1}$ that have a vertex $r$
in common. Root $T$ at $r$. Define $f(r):=d$ and $f(x):=d-1$ for every
other vertex $x$. Thus $F_1,\dots,F_{t-1}$ is a degree-$f$ covering of
the rooted tree $(T,r)$, and $\rmc[d](T)=\rmc[f](T,r)$.
\lemref{RootedMinCover} provides a recursive formula for
$\rmc[f](T,r)$, which implies (by the symmetry of $T$) that without
loss of generality, $r=v_0$. In particular, for each copy of $X$
rooted at some vertex $w$, every subtree in the induced covering of
$X$ contains $w$.

Fix $i\in\{-n,\dots,n\}$. Let $E_i$ be the set of $\Delta-2$ edges in
$T-E(P)$ incident to $v_i$. For each edge $v_iw\in E_i$, at least
$\pi$ of the subtrees $F_1,\dots,F_{t-1}$ intersect the copy of $X$
rooted at $w$, and each such subtree contains $w$. Since $f(w)=d-1$
and each such subtree $F_j$ is maximal, the edge $vw_i$ is also in
$F_j$. Thus $\sum_j|F_j\cap E_i|\geq(\Delta-2)\pi$.  Say a subtree
$F_j$ \emph{hits} $v_i$ if $|F_j\cap E_i|\geq d-1$. If $v_i$ is hit by
no subtree, then $|F_j\cap E_i|\leq d-2$ for all $j$, implying
$$(t-1)(d-2)\geq\sum_{j=1}^{t-1}|F_j\cap E_i|\geq(\Delta-2)\pi\enspace.$$
This is a contradiction since $t<\frac{\Delta-2}{d-2}\pi+1$. Thus
$v_i$ is hit by at least one subtree.

Hence the total number of hits is at least $2n+1$. Since
$2n+1>2(t-1)$, some subtree $F_j$ hits at least three vertices, say
$v_a,v_b,v_c$ where $-n\leq a<b<c\leq n$. Since $F_j$ is connected,
$F_j$ contains the path $(v_a,v_{a+1},\dots,v_c)$. Thus $v_b$ has
degree at least $(d-1)+2$ in $F_j$, which contradicts the assumption
that $F_j$ has maximum degree at most $d$. Therefore at least $t$
subtrees are needed in every covering of $T$ by degree-$d$ subtrees.
\end{proof}

\thmref{Pathwidth} says that trees with bounded maximum degree and
bounded pathwidth admit coverings by a bounded number of degree-$d$
subtrees. In the case of $d=3$, we now prove a converse result for a
large class of trees.

\begin{proposition}
  \proplabel{PathwidthUpperBound} Let $T$ be a tree in which every
  non-leaf vertex has degree at least $4$. Then $T$ has pathwidth at
  most $\mc[3](T)$.
\end{proposition}

\begin{proof}
  We proceed by induction on $c:=\mc[3](T)$. If $c=1$ then no vertex
  has degree at least 4, and every non-leaf vertex has degree at least
  $4$, implying $T\cong K_2$, which has pathwidth $1$. Now assume that
  $c\geq2$. Fix a covering of $T$ by $c$ degree-3 subtrees
  $T_1,\dots,T_c$. Let $S:=\bigcap_{i=1}^c T_i$. Since $T_1$ has
  maximum degree at most $3$, $S$ has maximum degree at most
  $3$. Suppose that $\deg_S(v)=3$ for some vertex $v$. By assumption,
  $\deg_T(v)\geq4$, implying there is an edge $vw\not\in S$. Since
  $vw\in E(T_i)$ for some $i$, we have $\deg_{T_i}(v)\geq4$, and $T_i$
  is not degree-3. This contradiction proves that $\deg_S(v)\leq 2$
  for every vertex $v$. Since $T_1$ is connected, $S$ is
  connected. Thus $S$ is a path. For each edge $vw$ of $T$ such that
  $v\in V(S)$ and $w\not\in V(S)$, let $T_w$ be the subtree of
  $T-V(S)$ that contains $w$. Since $vw\not\in E(S)$, at most $c-1$ of
  the subtrees $T_1,\dots,T_c$ contain $vw$. Since each such subtree
  is connected, $T_w$ is covered by at most $c-1$ subtrees. That is,
  $\mc[3](T_w)\leq c-1$. By induction, the pathwidth of $T_w$ is at
  most $c-1$. Therefore the pathwidth of $T$ is at most $c$.
\end{proof}

\thmref{Pathwidth} and \propref{PathwidthUpperBound} together say that
pathwidth is the right parameter to study when considering coverings
of trees by a bounded number of degree-3 subtrees.

\section{Coverings of General Graphs}
\seclabel{General}

This section considers coverings of general graphs by connected
subgraphs of bounded degree. 

A \emph{connected vertex cover} of a
graph $G$ is a connected subgraph $H$ of $G$ such that every edge of
$G$ has at least one endpoint in $H$; that is,
$E(G-V(H))=\emptyset$. For algorithmic aspects of connected vertex
covers, see \citep{MRR08,EGM-JDA10,FD04,GNW-TCS07}.

\begin{lemma}
  \lemlabel{General} Let $H$ be a connected vertex cover of a graph
  $G$. Let $$k:=\Delta(G-E(H))\enspace.$$ Then for every integer
  $d\geq\Delta(H)+1$, there is a covering of $G$
  by $$\CEIL{\frac{k+1}{d-\Delta(H)}}$$ connected degree-$d$
  subgraphs.
\end{lemma}

\begin{proof}
  By Vizing's Theorem \citep{Vizing64} applied to $G-E(H)$, there is a
  partition $\{E_i:1\leq i\leq k+1\}$ of $E(G)-E(H)$, such that each
  $E_i$ is a matching in $G-E(H)$. Grouping the matchings gives a
  partition $\{F_j:1\leq j\leq \Ceil{\frac{k+1}{d-\Delta(H)}}\}$ of
  $E(G)-E(H)$, such that each $F_j$ is a degree-$(d-\Delta(H))$
  subgraph of $G-E(H)$. Thus $H\cup F_j$ is a connected degree-$d$
  subgraph of $G$, and $\{H\cup F_j:1\leq j\leq
  \Ceil{\frac{k+1}{d-\Delta(H)}}\}$ is the desired covering of $G$.
\end{proof}


\begin{corollary}
  \corlabel{Spanning} Let $H$ be a connected spanning subgraph of a
  graph $G$. Then for every integer $d\geq\Delta(H)+1$, there is a
  covering of $G$
  by $$\CEIL{\frac{\Delta(G)-\delta(H)+1}{d-\Delta(H)}}$$ connected
  degree-$d$ subgraphs.
\end{corollary}

\begin{proof}
  The result follows from \lemref{General} with
  $k\leq\Delta(G)-\delta(H)$.
\end{proof}


\begin{corollary}
  \corlabel{Hamiltonian} For every integer $d\geq3$, every Hamiltonian
  graph $G$ has a covering by $$\CEIL{\frac{\Delta(G)-1}{d-2}}$$
  connected degree-$d$ subgraphs.
\end{corollary}

\begin{proof}
  Apply \corref{Spanning} with a Hamiltonian cycle $H$ of $G$.  Then
  $\Delta(H)=\delta(H)=2$. The result follows.
\end{proof}

This result can be slightly strengthened for $d=4$.

\begin{proposition}
  \proplabel{DegreeFourHamiltonian} For all $\epsilon>0$ there is an
  integer $\Delta_0$ such that every Hamiltonian graph $G$ with
  $\Delta(G)\geq\Delta_0$ has a covering
  by $$\CEIL{(\half+\epsilon)(\Delta(G)-2)}$$ connected degree-$4$
  subgraphs.
\end{proposition}

\begin{proof}
  A forest is \emph{linear} if each component is path.  The
  \emph{linear arboricity} of a graph $G$ is the minimum number of
  linear forests that partition $E(G)$.  \citet{Alon-IJM88} proved
  that for all $\epsilon>0$ there is an integer $\Delta_0$ such that
  every graph $G$ with $\Delta(G)\geq\Delta_0$ has linear arboricity
  at most $\CEIL{(\half+\epsilon)\Delta(G)}$.  Apply this result to
  $G-E(C)$ where $C$ is a Hamiltonian cycle in $G$. We obtain a
  partition \F\ of $E(G)-E(C)$ into
  $\CEIL{(\half+\epsilon)(\Delta(G)-2)}$ linear forests. Thus $\{C\cup
  F:F\in\F\}$ is a covering of $G$ by degree-4 subtrees.
\end{proof}

Now consider coverings of planar graphs by connected subgraphs of
bounded degree.  \citet{Tutte56} proved that every $4$-connected
planar graph is Hamiltonian. Thus \corref{Hamiltonian} implies the
next result.

\begin{corollary}
  \corlabel{FourConnectedPlanar} For every integer $d\geq3$, every
  $4$-connected planar graph $G$ has a covering
  by $$\CEIL{\frac{\Delta(G)-1}{d-2}}$$ connected degree-$d$
  subgraphs.\qed
\end{corollary}


\begin{corollary}
  \corlabel{ThreeConnectedPlanar} For every integer $d\geq4$, every
  $3$-connected planar graph $G$ has a covering
  by $$\CEIL{\frac{\Delta(G)}{d-3}}$$ connected degree-$d$ subgraphs.
\end{corollary}

\begin{proof}
  \citet{Barnette66} proved that $G$ has a degree-$3$ spanning tree
  $H$. The result follows from \corref{Spanning} with $\Delta(H)=3$
  and $\delta(H)=1$.
\end{proof}

Note that various generalisations of the above-mentioned result by
\citet{Barnette66} for graphs embedded on surfaces
\citep{Thom-JCTB94,Yu97,KNAO-JCTB03,SZ-JCTB98,GW-GC94,BEGMR-JCTB95,EG-JCTB94}
can be applied to obtain similar results to
\corref{ThreeConnectedPlanar}. We omit the details.


We conclude with an open problem: Is there a function $f$ and
constants $c$ and $d$ such that every $c$-connected graph $G$ has a
covering by $f(\Delta(G))$ connected degree-$d$ subgraphs?  We now
show that the answer is negative for $c=2$ and $d=2$ (even for
outerplanar graphs).

\begin{proposition}
  \proplabel{Arms} For all $k\geq2$ there is a $2$-connected
  outerplanar graph with maximum degree $3$ that requires at least $k$
  subgraphs in every covering by degree-$2$ connected subgraphs.
\end{proposition}

\begin{proof}
  Let $m:=2k$ and $n:=4k$.  Let $H$ be the graph obtained from
  disjoint paths $(a_1,\dots,a_m)$ and $(b_1,\dots,b_m)$ by adding the
  edge $a_ib_i$ for all $i\in[1,m]$. Each edge $a_ib_i$ is called a
  \emph{cross edge}, and $a_1b_1$ is called the \emph{base edge}.  As
  illustrated in \figref{Arms}, let $G$ be the graph obtained from a
  cycle $(v_1,\dots,v_{2n})$ and $n$ copies $H_1,\dots,H_{n}$ of $H$
  by identifying the base edge of $H_j$ with the edge $v_{2j-1}v_{2j}$
  for each $j\in[1,n]$. Observe that $G$ is 2-connected and
  outerplanar, and has maximum degree 3.  Let $X_1,\dots,X_t$ be a
  covering of $G$ by connected degree-$2$ subgraphs (that is, paths
  and cycles). To complete the proof we now show that $t\geq k$. Say
  $X_i$ \emph{occupies} $H_j$ if $X_i$ contains at least two cross
  edges in $H_j$. Observe that if $X_i$ occupies $H_j$, then either
  $X_i$ is a cycle contained in $H_j$, or $X_i$ is a path and it has
  an endpoint in $H_j$. Thus each $X_i$ occupies at most two $H_j$
  subgraphs, implying $X_i$ contains less than $n+2m$ cross
  edges. Since there are $nm$ cross edges in total,
  $t>\frac{nm}{n+2m}=k$.
\end{proof}

\Figure{Arms}{\includegraphics{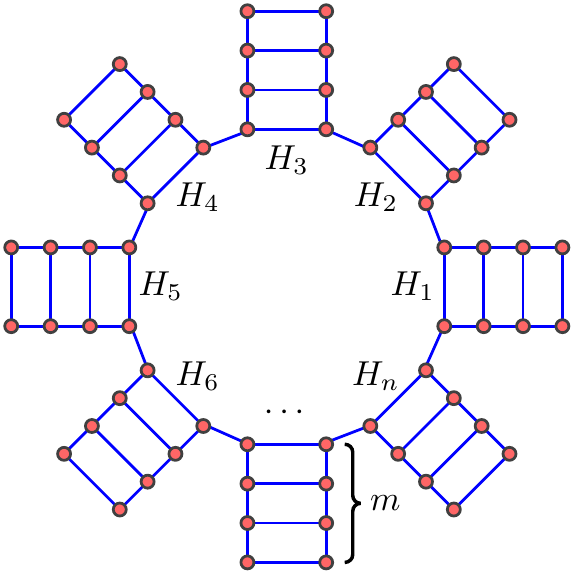}}{Construction in the proof of
  \propref{Arms}.}

This question seems related to a result by \citet{CXY04}, who proved
that every $3$-connected graph $G$ with $n\geq 4$ vertices and maximum
degree at most $d\geq 3$ contains a cycle of length at least
$n^{\log_b 2}+2$, where $b=2(d-1)^2+1$.


\def\cprime{$'$} \def\soft#1{\leavevmode\setbox0=\hbox{h}\dimen7=\ht0\advance
  \dimen7 by-1ex\relax\if t#1\relax\rlap{\raise.6\dimen7
  \hbox{\kern.3ex\char'47}}#1\relax\else\if T#1\relax
  \rlap{\raise.5\dimen7\hbox{\kern1.3ex\char'47}}#1\relax \else\if
  d#1\relax\rlap{\raise.5\dimen7\hbox{\kern.9ex \char'47}}#1\relax\else\if
  D#1\relax\rlap{\raise.5\dimen7 \hbox{\kern1.4ex\char'47}}#1\relax\else\if
  l#1\relax \rlap{\raise.5\dimen7\hbox{\kern.4ex\char'47}}#1\relax \else\if
  L#1\relax\rlap{\raise.5\dimen7\hbox{\kern.7ex
  \char'47}}#1\relax\else\message{accent \string\soft \space #1 not
  defined!}#1\relax\fi\fi\fi\fi\fi\fi} \def\Dbar{\leavevmode\lower.6ex\hbox to
  0pt{\hskip-.23ex\accent"16\hss}D}

\appendix

\section{Complete Multipartite Graphs}

For a graph $G$ and integer $d\geq1$, let $\cliques[d](G)$ be the
minimum number of disjoint $(\leq d)$-cliques of $G$ that partition
$V(G)$. For example, $\cliques[1](G)=|V(G)|$, and
$\cliques[2](G)=|V(G)|-p$, where $p$ is the number of edges in a
maximum matching in $G$. Let $K\ang{n_1,\dots,n_k}$ be the complete
$k$-partite graph with $n_i$ vertices in the $i$-th colour class. We
now determine $\cliques[d](K\ang{n_1,\dots,n_k})$.

\begin{lemma}
  \lemlabel{CMG} For all integers $k\geq d\geq1$ and
  $n_1,\dots,n_k\geq0$,
$$\cliques[d](K\ang{n_1,\dots,n_k})=
\MAXM{\max_{1\leq i\leq
    k}n_i,\CEIL{\frac{1}{d}\sum_{i=1}^kn_i}}\enspace.$$ Moreover,
there is a $O(\sum_{i=1}^kn_i)$ time algorithm to compute a partition
of $K\ang{n_1,\dots,n_k}$ into this many $(\leq d)$-cliques.
\end{lemma}

\begin{proof}
  Since each vertex in the $i$-th colour class is in a distinct clique
  of the partition, $\cliques[d](K\ang{n_1,\dots,n_k})\geq n_i$. Since
  every vertex is in some clique of the partition,
  $d\cdot\cliques[d](K\ang{n_1,\dots,n_k})\geq\sum_{i=1}^kn_i$. Thus
  proves the lower bound on $\cliques[d](K\ang{n_1,\dots,n_k})$.

  It remains to prove the upper bound. We proceed by induction on
  $d+\sum_{i=1}^kn_i$. Assume that $n_1\geq \dots\geq n_k\geq0$.  If
  $d=1$ then
$$\cliques[1](K\ang{n_1,\dots,n_k})=\sum_{i=1}^kn_i=\MAXM{n_1,\sum_{i=1}^kn_i},$$ 
as desired. Now assume that $d\geq2$. First suppose that $n_d=0$. Then
\begin{align*}
  \cliques[d](K\ang{n_1,\dots,n_k})=
  \cliques[d-1](K\ang{n_1,\dots,n_{d-1}}),
\end{align*}
and by induction
\begin{align*}
  \cliques[d](K\ang{n_1,\dots,n_k})
  \leq\MAXM{n_1,\CEIL{\frac{1}{d-1}\sum_{i=1}^{d-1}n_i}} =n_1
  \leq\MAXM{n_1,\CEIL{\frac{1}{d}\sum_{i=1}^{k}n_i}}\enspace.
\end{align*}
Now assume that $n_d\geq1$. Let $C$ be a set with exactly one vertex
from each of the $d$ largest colour classes. So $C$ is a $d$-clique,
and
\begin{equation}
  \eqnlabel{Induction}
  \cliques[d](K\ang{n_1,\dots,n_k}) \leq 
  1+\cliques[d](K\ang{n_1-1,\dots,n_d-1,n_{d+1},\dots,n_k})\enspace.
\end{equation}
Suppose that $n_1=\dots=n_{d+1}$ ($\geq1$).  Thus by
\eqnref{Induction} and induction
\begin{align*}
  \cliques[d](K\ang{n_1,\dots,n_k})
  &\;\leq\; 1+\MAXM{n_{d+1},\CEIL{\frac{1}{d}\BRACKET{\BRACKET{\sum_{i=1}^{k}n_i}-d}}}\\
  &=\MAXM{1+n_1,\CEIL{\frac{1}{d}\sum_{i=1}^kn_i}}\enspace.
\end{align*}
Observe that
$$\CEIL{\frac{1}{d}\sum_{i=1}^kn_i}\geq\CEIL{\frac{(d+1)n_1}{d}}
=n_1+\CEIL{\frac{n_1}{d}}\geq n_1+1\enspace.$$ Thus
$$\cliques[d](K\ang{n_1,\dots,n_k}) \leq 
\CEIL{\frac{1}{d}\sum_{i=1}^kn_i}
=\MAXM{n_1,\CEIL{\frac{1}{d}\sum_{i=1}^kn_i}},$$ as desired.  Now
assume that $n_{d+1}<n_d$.  Hence by \eqnref{Induction} and induction,
\begin{align*}
  \cliques[d](K\ang{n_1,\dots,n_k}) &\;\leq\;
  1+\MAXM{n_1-1,\CEIL{\frac{1}{d}\BRACKET{\BRACKET{\sum_{i=1}^kn_i}-d}}}\\
  &\;=\; \MAXM{n_1,\CEIL{\frac{1}{d}\sum_{i=1}^kn_i}},
\end{align*}
as desired. It is easily seen that this proof can be adapted to give a
greedy linear-time algorithm to compute the partition, where at each
stage, a $d$-clique is repeatedly chosen from the $d$ largest colour
classes.
\end{proof}

Note that the case $d=2$ in \lemref{CMG} also follows from a result by
\citet{Sitton-EJUM}, who determined the size of the largest matching
in $K\ang{n_1,\dots,n_k}$.

The \emph{Turan} graph $K\ang{n;k}$ is the complete $k$-partite graph
$K\ang{n_1,\dots,n_k}$ where $n=\sum_{i=1}^kn_i$ and $|n_i-n_j|\leq1$
for $i,j\in\{1,\dots,k\}$.

\begin{corollary}
  \corlabel{Turan} For all integers $k\geq d\geq1$,
$$\cliques[d](K\ang{n;k})=\CEIL{\frac{n}{d}}\enspace.$$
Moreover, there is $O(n)$ time algorithm to compute a partition of
$K\ang{n;k}$ into $\Ceil{\frac{n}{d}}$ $(\leq d)$-cliques.
\end{corollary}

\begin{proof}
  Let $x$ and $y$ be integers such that $n=xk+y$ where $0\leq y\leq
  k-1$.  Then $$K\ang{n;k}\cong
  K\ang{\underbrace{x,\dots,x}_{k-y},\underbrace{x+1,\dots,x+1}_y}\enspace.$$
  By \lemref{CMG},
$$\cliques[d](K\ang{n;k})=
\begin{cases}
  \max\{x,\Ceil{\frac{n}{d}}\}	& \text{ if }y=0\enspace,\\
  \max\{x+1,\Ceil{\frac{n}{d}}\} & \text{ if }y\geq 1\enspace.
\end{cases}$$ If $y=0$ then $n=xk\geq xd$ and $\frac{n}{d}\geq x$.  If
$y\geq 1$ then $n\geq xk+1\geq xd+1$ and $\Ceil{\frac{n}{d}}\geq x+1$.
In both cases, $\cliques[d](K\ang{n;k})=\Ceil{\frac{n}{d}}$.
\end{proof}

In the proof of the upper bound in \thmref{RootedPathwidth}, we need
the following result about partitioning Turan graphs into cliques of
two specified sizes.

\begin{corollary}
  \corlabel{CMGCMG} For all integers $n,k,p,q,m$, such that $n,k\geq 1$
  and $k\geq p,q\geq 0$ and $m\geq0$, there is a vertex
  partition of the Turan graph $K\ang{n;k}$ into $m$ $(\leq
  p)$-cliques and $\Ceil{\frac{\max\{n-mp,0\}}{q}}$ $(\leq
  q)$-cliques.
\end{corollary}

\begin{proof}
  We proceed by induction on $m$. If $m=0$ then by \corref{Turan},
  $K\ang{n;k}$ has a partition into $\ceil{\frac{n}{q}}$ $(\leq
  q)$-cliques, as desired. Now assume that $m\geq1$. 
If $n\leq p$ then $K\ang{n;k}\cong K_n$ has a partition into one
$(\leq p)$-clique and zero $(\leq q)$-cliques, as desired. 
  Now assume that $n\geq p+1$. Let $C$ be a $p$-clique with exactly one
  vertex from each of the $p$ largest colour classes of
  $K\ang{n;k}$. This is well-defined since $p\leq k$. 
Then $K\ang{n;k}-C\cong K\ang{n-p;k}$. By induction,
  there is a vertex partition of $K\ang{n-p;k}$ into $m-1$ $(\leq
  p)$-cliques and $\Ceil{\frac{n-p-(m-1)p}{q}}$ $(\leq q)$-cliques.
  Since $\Ceil{\frac{n-p-(m-1)p}{q}}=\Ceil{\frac{n-mp}{q}}$, with $C$,
  we have a vertex partition of $K\ang{n;k}$ into $m$ $(\leq p)$-cliques
  and $\Ceil{\frac{n-mp}{q}}$ $(\leq q)$-cliques.
\end{proof}

\section{Integer Linear Programs}

This appendix contains a solution to the integer linear program that arose in the proof of the lower bound in \thmref{RootedPathwidth}.

\begin{lemma}
\lemlabel{IntegerProgram}
Fix integers $A,B\geq1$ and $d\geq2$ such that $A\leq B$ and $(d-2)B\leq(d-1)A$. Suppose that some non-negative integers $x,y_1,y_2,z$ satisfy 
\begin{align}
(d-2)x+(d-1)(y_1+y_2)+dz	& \geq A	\eqnlabel{Middle}\\
(d-1)(x+y_1)				& \geq B	\eqnlabel{Left}\\
(d-1)(x+y_2)				& \geq B.	\eqnlabel{Right}
\end{align}
Then $$x+y_1+y_2+z\;\geq\;
\CEIL{\frac{A}{d}+\frac{2}{d}\CEIL{\frac{B}{d-1}}}\enspace.$$
This bound is achievable, for example by 
$$x:=\CEIL{\frac{B}{d-1}},\;\;y_1:=y_2:=0,\;\; z:=\CEIL{\frac{A-x(d-2)}{d}}\enspace.$$
\end{lemma}

\begin{proof}
Say $(x,y_1,y_2,z)$ is a \emph{solution} if \threeeqnref{Middle}{Left}{Right} are satisfied. A solution is \emph{optimal} if it minimises $x+y_1+y_2+z$.
Suppose that $(x,y_1,y_2,z)$ is a solution, where $y_1\geq y_2$.
We claim that $(x+y_2,0,0,z+y_1)$ is also a solution.
By \eqnref{Middle}  and since $y_1\geq y_2$,
\begin{align*}
		& (d-2)(x+y_2)\,+(d-1)(0+0)\,+\,d(z+y_1)\\	
=\;		& (d-2)x\,+\,(d-2)y_2\,+\,dy_1\,+\,dz\\	
\geq\; 	& (d-2)x+(d-1)(y_1+y_2)+dz\\
\geq\;	& A\enspace.
\end{align*}
Thus $(x+y_2,0,0,z+y_1)$ satisfies \eqnref{Middle}.
By \eqnref{Right} and since $y_1\geq y_2$,
$$(d-1)(x+y_1)	\geq	(d-1)(x+y_2)	 \geq B\enspace.$$	
Thus $(x+y_2,0,0,z+y_1)$ satisfies \eqnref{Left} and \eqnref{Right}.
Hence $(x+y_2,0,0,z+y_1)$ is also a solution, as claimed.
Since $$x+y_1+y_2+z= (x+y_2)+0+0+(z+y_1),$$ 
there is an optimal solution $(x,0,0,z)$.
By \eqnref{Left} or \eqnref{Right}, 
$$x=x_i:=\CEIL{\frac{B}{d-1}}+i$$
for some integer $i\geq0$. By \eqnref{Middle},
$$z\geq z_i:=\CEIL{\frac{A-(d-2)x_i}{d}}\enspace.$$
Now,
\begin{align*}
x+z\;\geq\; x_i+z_i \;=\; 
\CEIL{\frac{A+2x_i}{d}} \;\geq\;
\CEIL{\frac{A+2x_0}{d}} \;=\; 
x_0+z_0\enspace.
\end{align*}
Thus if $(x_0,0,0,z_0)$ is a solution, then it is optimal. 
Thus it suffices to prove that $(x_0,0,0,z_0)$ is a solution.
Clearly $x_0\geq 0$ and \threeeqnref{Middle}{Left}{Right} are satisfied.
It remains to prove that $z_0\geq 0$. 
We have 
\begin{align*}
(d-2)x_0
=(d-2)\CEIL{\frac{B}{d-1}}
\leq\frac{(d-2)B}{d-1}+(d-2)
\leq A+d-2\enspace.
\end{align*}
Thus $A-(d-2)x_0\geq2-d$ and
\begin{align*}
z_0=\CEIL{\frac{A-(d-2)x_0}{d}}
\geq \CEIL{\frac{2-d}{d}}
=\CEIL{\frac{2}{d}}-1=0\enspace,
\end{align*}
as desired. Hence 
$(x_0,0,0,z_0)$ is an optimal solution.
The claimed lower bound on $x+y_1+y_2+z$ follows by substitution.
\end{proof}

Now we solve another integer program that is needed in the proof of the lower bound in \thmref{RootedPathwidth} with $\Delta=d+1$. 

\begin{lemma}
\lemlabel{AnotherIntegerProgram}
Fix integers $A\geq1$ and $d\geq2$. Let $r:=A\bmod{d(d-1)}$. Thus $0\leq r\leq d(d-1)-1$. Suppose that $x,y_1,y_2,z\in\Z$ satisfy 
\begin{align}
(d-2)x+(d-1)(y_1+y_2+z)		& \geq (d-1)A	\eqnlabel{AnotherMiddle}\\
(d-1)(x+y_1)				& \geq dA		\eqnlabel{AnotherLeft}\\
(d-1)(x+y_2)				& \geq dA.		\eqnlabel{AnotherRight}
\end{align}
Then $$x+y_1+y_2+z\;\geq\;
\begin{cases}
\CEIL{\frac{d-1}{d-2}A-\frac{2}{d-2}\FLOORFRAC{A}{d(d-1)}}	
& \text{if }r<(d-1)(d-2)\\
\CEIL{\frac{d}{d-1}A+\CEILFRAC{A}{d(d-1)}}
& \text{if }r\geq(d-1)(d-2)\enspace.\\
\end{cases}$$
This bound is achievable, for example by 
\begin{align*}
x&:=
\begin{cases}
\CEIL{\frac{d-1}{d-2}A-\frac{2d-2}{d-2}\FLOORFRAC{A}{d(d-1)}}
& \text{if }r<(d-1)(d-2)\\
\CEIL{\frac{d}{d-1}A-\CEILFRAC{A}{d(d-1)}}
& \text{if }r\geq(d-1)(d-2)\enspace,\\
\end{cases}\\
y_1:=y_2&:=\begin{cases}
\FLOORFRAC{A}{d(d-1)}	& \text{if }r<(d-1)(d-2)\\
\CEILFRAC{A}{d(d-1)}	& \text{if }r\geq(d-1)(d-2)\enspace,\\
\end{cases}\\
z&:=0\enspace.
\end{align*}
\end{lemma}

\begin{proof}
Say $(x,y_1,y_2,z)$ is a \emph{solution} if \threeeqnref{AnotherMiddle}{AnotherLeft}{AnotherRight} are satisfied. A solution is \emph{optimal} if it minimises $x+y_1+y_2+z$. Observe that if $(x,y_1,y_2,z)$ is an optimal solution, then $(x,y_1+z,y_2,0)$ also is an optimal solution. Thus $(x,y_1,y_2,0)$ is an optimal solution for some $y_1\geq y_2$. 

Let $(x,y_1,y_2,0)$ be an optimal solution with $y_1\geq y_2$, such that $y_1-y_2$ is minimised. Suppose on the contrary that $y_1-y_2\geq 1$. Then $(x-1,y_1,y_2+1,0)$ is a solution since 
\begin{align*}
(d-2)(x-1)+(d-1)(y_1+y_2+1)
&=(d-2)x+(d-1)(y_1+y_2)-(d-2)+(d-1)\\
&=(d-2)x+(d-1)(y_1+y_2)+1\\
&\geq (d-1)A+1,
\end{align*}
and by \eqnref{AnotherRight},
\begin{align*}
(d-1)(x-1+y_1)&\geq(d-1)(x+y_2)\geq dA,\text{ and}\\
(d-1)(x-1+y_2+1)&=(d-1)(x+y_2)\geq dA.
\end{align*}
Moreover, $(x-1,y_1,y_2+1,0)$ is optimal since
$x+y_1+y_2=(x-1)+y_1+(y_2+1)$. This proves that $(x,y_1,y_2,0)$ does not minimise $y_1-y_2$, which is a contradiction. Thus $y_1=y_2$. 

Hence $(x^*,y^*,y^*,0)$ is an optimal solution for some $x^*,y^*\in\Z$. That is, $x^*$ and $y^*$ minimise $x^*+2y^*$ such that 
\begin{align}
(d-2)x^*+2(d-1)y^*		& \geq (d-1)A		\eqnlabel{MiddleMiddle}\\
(d-1)(x^*+y^*)			& \geq dA\enspace.	\eqnlabel{RightRight}
\end{align}

First consider the case when $d=2$. Then \eqnref{MiddleMiddle} and \eqnref{RightRight} hold if and only if $y^*\geq\CeilFrac{A}{2}$ and $x^*\geq2A-y^*$. Thus $x^*+2y^*$ is minimised by 
$y^*=\CeilFrac{A}{2}$ and $x^*=2A-y^*=\FloorFrac{3A}{2}$. Hence
$x^*+2y^*=\FloorFrac{3A}{2}+2\CeilFrac{A}{2}=\CeilFrac{5A}{2}$. This result matches the claimed bounds since $r\geq(d-1)(d-2)$ when $d=2$.

Now assume that $d\geq3$. Thus \eqnref{MiddleMiddle} and \eqnref{RightRight} hold if and only if 
\begin{align*}
x^*\geq
\CEIL{\MAXM{\frac{d-1}{d-2}A-\frac{2d-2}{d-2}y^*,\frac{d}{d-1}A-y^*}}\enspace.
\end{align*}
That is, 
\begin{equation*}
x^*+2y^*\geq
\CEIL{\MAXM{\frac{d-1}{d-2}A-\frac{2}{d-2}y^*,\frac{d}{d-1}A+y^*}}\enspace.
\end{equation*}
Define 
\begin{align*}
f(y)
:=\MAXM{\frac{d-1}{d-2}A-\frac{2}{d-2}y,\frac{d}{d-1}A+y}
=\begin{cases}
\frac{d-1}{d-2}A-\frac{2}{d-2}y
	& \text{ if }y \leq \frac{A}{d(d-1)}\\
\frac{d}{d-1}A+y
	& \text{ if }y\geq\frac{A}{d(d-1)}\enspace.
\end{cases}
\end{align*}
For a given value of $y^*$, setting $$x^*:=\CEIL{\MAXM{\frac{d-1}{d-2}A-\frac{2d-2}{d-2}y^*,\frac{d}{d-1}A-y^*}}$$
implies that that $x^*+2y^*=\CEIL{f(y^*)}$. Since $x^*$ and $y^*$ minimise $x^*+2y^*$,
\begin{align*}
x^*+2y^*
=\min_{y\in\Z}
\CEIL{f(y)}
=\min_{y\in\Z}
\begin{cases}
\CEIL{\frac{d-1}{d-2}A-\frac{2}{d-2}y}
	& \text{ if }y \leq \FLOORFRAC{A}{d(d-1)}\\
\CEIL{\frac{d}{d-1}A+y}
	& \text{ if }y\geq \CEILFRAC{A}{d(d-1)}\enspace.
\end{cases}
\end{align*}
Observe that $\frac{A}{d(d-1)}$ is the only local minimum (and thus the global minimum) of $f$. Hence
$$y^*=\FLOORFRAC{A}{d(d-1)}\text{ or }y^*=\CEILFRAC{A}{d(d-1)}\enspace.$$
If $r=0$ then $y^*=\frac{A}{d(d-1)}$ is an integer, and we are done with $x^*=\frac{d+1}{d}A$ and $x^*+2y^*=\frac{d^2+1}{d^2-d}A$. 
Now assume that $r\geq1$. Thus
\begin{align*}
&&f\BRACKET{\FLOORFRAC{A}{d(d-1)}}&< f\BRACKET{\CEILFRAC{A}{d(d-1)}}\\
&\Longleftrightarrow&
\frac{d-1}{d-2}A-\frac{2}{d-2}\FLOORFRAC{A}{d(d-1)}
&< \frac{d}{d-1}A+\CEILFRAC{A}{d(d-1)}\\
&\Longleftrightarrow&
\frac{d-1}{d-2}A-\frac{2}{d-2}\cdot\frac{A-r}{d(d-1)}
&< \frac{d}{d-1}A+\frac{A-r}{d(d-1)}+1\\
&\Longleftrightarrow&
r&< (d-1)(d-2)\enspace.
\end{align*}
Thus 
$$x^*+2y^*=\CEIL{f(y^*)}=
\begin{cases}
\CEIL{\frac{d-1}{d-2}A-\frac{2}{d-2}\FLOORFRAC{A}{d(d-1)}}
	& \text{if }r<(d-1)(d-2)\\
\CEIL{\frac{d}{d-1}A+\CEILFRAC{A}{d(d-1)}}
	& \text{if }r\geq (d-1)(d-2)\enspace,
\end{cases}$$
where
$$y^*=
\begin{cases}
\FLOORFRAC{A}{d(d-1)}	& \text{if }r< (d-1)(d-2)\\
\CEILFRAC{A}{d(d-1)}	& \text{if }r\geq (d-1)(d-2)\enspace
\end{cases}$$
and 
$$x^*=
\begin{cases}
\CEIL{\frac{d-1}{d-2}A-\frac{2d-2}{d-2}\FLOORFRAC{A}{d(d-1)}}	
& \text{if }r<(d-1)(d-2)\\
\CEIL{\frac{d}{d-1}A-\CEILFRAC{A}{d(d-1)}}
& \text{if }r\geq(d-1)(d-2)\enspace.
\end{cases}$$
\end{proof}

\end{document}